\documentclass[reqno,12pt]{amsart}

\usepackage{amsmath, epsfig, cite}
\usepackage{amssymb}
\usepackage{amsfonts}
\usepackage{amsthm}
\usepackage{graphicx}
\usepackage{subfigure}
\usepackage{float}
\usepackage{url}
\usepackage{microtype}
\usepackage{xcolor}
\usepackage{tikz}
\usetikzlibrary{arrows.meta,calc}
\usepackage{hyperref}
\hypersetup{hidelinks,
    pdftitle={Minimal Covering Bodies and a Minkowski-Type Criterion for Lattice Coverings},
    pdfauthor={Yanlu Lian and Fei Xue}}

\newtheorem{theorem}{Theorem}[section]
\newtheorem*{theorem*}{Theorem}
\newtheorem{proposition}[theorem]{Proposition}
\newtheorem{corollary}[theorem]{Corollary}
\newtheorem{conjecture}[theorem]{Conjecture}
\newtheorem{lemma}[theorem]{Lemma}
\newtheorem{problem}[theorem]{Problem}

\theoremstyle{definition}
\newtheorem{definition}[theorem]{Definition}

\theoremstyle{remark}
\newtheorem{remark}[theorem]{Remark}

\numberwithin{equation}{section}
\allowdisplaybreaks[1]

\newcommand{\conv}{\mathrm{conv}\,}

\newcommand{\va}{{\mathbf a}}
\newcommand{\vb}{{\mathbf b}}
\newcommand{\vc}{{\mathbf c}}
\newcommand{\vd}{{\mathbf d}}
\newcommand{\ve}{{\mathbf e}}
\newcommand{\vf}{{\mathbf f}}
\newcommand{\vg}{{\mathbf g}}
\newcommand{\vh}{{\mathbf h}}
\newcommand{\vu}{{\mathbf u}}
\newcommand{\vv}{{\mathbf v}}
\newcommand{\vw}{{\mathbf w}}
\newcommand{\vx}{{\mathbf x}}
\newcommand{\vy}{{\mathbf y}}
\newcommand{\vz}{{\mathbf z}}
\newcommand{\vnull}{{\mathbf 0}}
\newcommand{\inte}{\mathrm{int}\,}

\newcommand{\bd}{\mathrm{bd}\,}

\newcommand{\Rr}{{\mathbb R}}
\newcommand{\Z}{{\mathbb Z}}
\newcommand{\vol}{\mathrm{vol}\,}

\begin{document}

\title[Minimal Covering Bodies]{Minimal Covering Bodies and a Minkowski-Type Criterion for Lattice Coverings}

\author{Yanlu Lian}
\address{School of Mathematics, Hangzhou Normal University, Hangzhou, 311121, China}
\email{yllian@hznu.edu.cn}

\author{Fei Xue}
\address{School of Mathematical Sciences, Nanjing Normal University, Nanjing, 210046, China}
\email{05429@njnu.edu.cn}



\thanks{This work is supported by the National Natural Science Foundation of China (NSFC12201307), the Natural Science Foundation of Zhejiang Province (LQ24A010008) and the Scientific Research Fund of Zhejiang Provincial Education Department (Y202351675).}
\date{September 11, 2026}

\begin{abstract}
We study convex bodies whose translates by a fixed lattice cover space
and for which every proper convex subbody loses this property. In three
dimensions, the convex hull of independent translates of the six Kuhn
tetrahedra always gives a lattice covering. We give a geometric proof
using auxiliary tetrahedra and the parity of the covering multiplicity.
We also prove polytopality and boundary restrictions for minimal covering
bodies, and show that every three-dimensional parallelohedron admits a
Kuhn representation with respect to its face-to-face tiling lattice.
Constructions from Reeve tetrahedra give non-symmetric minimal covering
bodies with eight vertices in dimension three and centrally symmetric ones with
sixteen vertices in dimension four, with unbounded volumes for the
integer covering lattice. They have pairwise distinct arithmetic contact
types, which record lattice contacts and the faces containing them.
The covering theorem also gives
a finite intersection criterion for a prescribed lattice basis, with
three intersection tests in the centrally symmetric case.
\end{abstract}

\maketitle

\section{Introduction}

Lattice coverings are a basic subject in the geometry of numbers. Let
$\mathcal K^n$ be the family of compact convex subsets of $\mathbb R^n$
with nonempty interior. For $K\in\mathcal K^n$ and a full-rank lattice
$\Lambda\subset\mathbb R^n$, the translates $K+\Lambda$ form a
\textbf{lattice covering} if $K+\Lambda=\mathbb R^n$. The optimal lattice
covering density is
\[
 \theta^l(K)=\inf\left\{\frac{\vol(K)}{\det\Lambda}:
                         K+\Lambda=\mathbb R^n\right\},
\]
where $\det\Lambda$ is the volume of a fundamental parallelepiped of
$\Lambda$. For $n\ge3$, general upper bounds include Rogers' estimate
$\theta^l(K)\le n^{\log_2\log_e n+c}$~\cite{Rogers1959} and the bound
$\theta^l(K)\le cn^2$ of Ordentlich, Regev and Weiss~\cite{Ordentlich},
with absolute constants. Exact values and the structure of the relevant
coverings remain difficult even for familiar three-dimensional bodies;
see \cite{Brass2005,Conway2003,Gruber1987,Rogers1964}.

For lattice packings, Minkowski's three-dimensional
criterion~\cite{Minkowski1904} gives a small list of possible contacts
between a convex body and a critical lattice. Betke and Henk
\cite{Betke2000} used this structure to determine the densest lattice
packings of three-dimensional polytopes. This motivates the search for
finite geometric conditions for lattice coverings. Our criterion tests
whether a body contains six prescribed simplex shapes associated with
one lattice basis. It is a sufficient condition; the classification
problem below asks which minimal covering bodies admit such a
representation. The lattice covering
of a tetrahedron with density $125/63$~\cite{Dougherty,Forcade2000}
provides one motivation for understanding such conditions.

A natural starting point is the theory of parallelohedra, that is,
convex polytopes which tile space by lattice translations. Their relation
to Dirichlet--Voronoi cells is the subject of Voronoi's
conjecture~\cite{Voronoi1908}; see \cite{Gruber1987,Zong2006,Zong2014}
for background. In the plane, the lattice covering problem is governed
by inscribed centrally symmetric hexagons~\cite{Dowker,Fary1950,Fejes1950}.
In higher dimensions, an inscribed lattice-covering body need not be a
parallelohedron or centrally symmetric, as shown by Xue and
Zong~\cite{Xue2017}.

These observations lead us to \textbf{minimal covering bodies}. For a
fixed lattice $\Lambda$, a convex body $M$ is minimal if
$M+\Lambda=\mathbb R^n$ and no proper convex subbody of $M$ has this
property. Thus minimality concerns inclusion with the lattice fixed;
it does not mean least volume or optimal covering density. Codenotti,
Freyer and Krivoku\'ca~\cite{Codenotti} independently studied the same
notion and established polytopality. We give a direct proof of
polytopality and indicate the corresponding results in their work.

Our first result is a sufficient condition for covering. The six
tetrahedra $\Delta_1,\ldots,\Delta_6$ in the standard Kuhn triangulation
of $[0,1]^3$ are listed in Section~\ref{sec:Kuhn simplices}.

\begin{theorem}\label{thm:kuhn_cover}
Let
\[
 K=\conv\bigcup_{i=1}^6(\Delta_i+\vx_i),
 \qquad \vx_i\in\mathbb R^3.
\]
Then $K+\mathbb Z^3=\mathbb R^3$.
\end{theorem}

The translations are independent: no agreement between the positions
of different tetrahedra is required. We construct auxiliary tetrahedra
contained in lattice translates of their convex hull. Pairing their
triangular faces shows that the covering multiplicity has constant odd
parity, which proves the covering assertion. The conclusion also follows from Balitskiy's
triangulation lemma~\cite[Lemma~5.3]{Balitskiy}, which applies in every
dimension. We state this connection at the end of
Section~\ref{sec:Kuhn simplices}.

For classification, the lattice contacts contain information that the
ordinary face lattice does not record. Definition~\ref{def:type}
introduces the \textbf{arithmetic contact type}: it records the lattice
contacts rooted at the vertices of the body, together with the faces
containing those contacts, under one common lattice isomorphism.
This definition depends on the body and its lattice, rather than on a
chosen decomposition into translated polytopes.

Let $m$ be a positive integer and put
\[
 T_m=\conv\{(0,0,0),(1,0,0),(0,1,0),(1,1,m)\},
 \qquad
 \widetilde T_m=T_m+[\mathbf0,\ve_3],
\]
where $\ve_3=(0,0,1)$ and the sum is a Minkowski sum.

\begin{theorem}\label{thm:infinite_asym}
For every positive integer $m$, $\widetilde T_m$ is a non-centrally
symmetric minimal covering body for $\mathbb Z^3$.
The bodies $\widetilde T_m$ have pairwise distinct arithmetic contact types.
\end{theorem}

The corresponding centrally symmetric bodies in dimension four are
\[
 \widetilde T_m^\bullet=
 \conv\bigl((\widetilde T_m\times\{0\})
              \cup(-\widetilde T_m\times\{1\})\bigr).
\]

\begin{theorem}\label{thm:infinite_sym}
For every positive integer $m$, $\widetilde T_m^\bullet$ is a centrally
symmetric minimal covering body for $\mathbb Z^4$.
The bodies $\widetilde T_m^\bullet$ have pairwise distinct arithmetic
contact types.
\end{theorem}

The two families have eight and sixteen vertices, respectively, and
the volume computations in Sections~\ref{sec:3-dim asy MBs}
and~\ref{sec:four-dimensional CSM} give
\[
 \vol_3(\widetilde T_m)=1+\frac m6,
 \qquad \vol_4(\widetilde T_m^\bullet)=1+\frac m3.
\]
Thus inclusion-minimality imposes no upper bound on
$\vol(K)/\det\Lambda$, even with the dimension and the number of
vertices fixed, and even under central symmetry in dimension four.
These densities use the specified covering lattices
$\mathbb Z^3$ and $\mathbb Z^4$, whereas $\theta^l(K)$ takes the
infimum over all covering lattices.
The infinitude in Theorems~\ref{thm:infinite_asym}
and~\ref{thm:infinite_sym} concerns arithmetic contact types.
The five types of three-dimensional parallelohedra~\cite{Federov},
by contrast, are ordinary combinatorial types.

Theorem~\ref{thm:kuhn_cover} suggests the following classification
problem. For a lattice basis $B=(\mathbf b_1,\mathbf b_2,\mathbf b_3)$
and a permutation $\sigma\in S_3$, write
\[
 \Delta_\sigma(B)=\conv\{\mathbf0,\mathbf b_{\sigma(1)},
 \mathbf b_{\sigma(1)}+\mathbf b_{\sigma(2)},
 \mathbf b_1+\mathbf b_2+\mathbf b_3\}.
\]
\begin{problem}\label{conj:sym_classification}
Characterize the centrally symmetric minimal covering bodies $K$ for
lattices $\Lambda\subset\mathbb R^3$ that admit a representation
\[
 K=\conv\bigcup_{\sigma\in S_3}
       \bigl(\Delta_\sigma(B)+\mathbf t_\sigma\bigr)
\]
for some basis $B$ of $\Lambda$ and vectors
$\mathbf t_\sigma\in\mathbb R^3$.
\end{problem}

The parallelohedra provide a natural class for this question.
In Section~\ref{sec:definition} we prove that every three-dimensional
parallelohedron has such a representation for its face-to-face tiling
lattice. The six simplices use one common lattice basis.
For a prescribed basis, Section~\ref{sec:discussion} expresses their
containment by six intersections of translates of $K$, reduced to
three intersections when $K$ is centrally symmetric.

The paper is organized as follows. Section~\ref{sec:Kuhn simplices}
proves the Kuhn covering theorem. Section~\ref{sec:definition} proves
polytopality, defines arithmetic contact type, and establishes boundary
restrictions. Sections~\ref{sec:3-dim asy MBs}
and~\ref{sec:four-dimensional CSM} treat the three- and four-dimensional constructions.
Section~\ref{sec:discussion} gives the finite intersection criterion
and its centrally symmetric specialization.

\section{Convex Hull of Kuhn Simplices}\label{sec:Kuhn simplices}

We prove Theorem~\ref{thm:kuhn_cover} by constructing tetrahedra contained
in lattice translates of $K$. Their triangular faces occur in pairs.
This pairing allows us to compare the covering with the initial
subdivision of the unit cube by counting modulo two.

Label the vertices of $[0,1]^3$ by
\begin{equation}\label{eq:kuhn-cube-vertices}
\begin{split}
 &\va=(0,0,0),\quad \vb=(1,0,0),\quad
 \vc=(1,1,0),\quad \vd=(0,1,0),\\
 &\ve=(0,0,1),\quad \vf=(1,0,1),\quad
 \vg=(1,1,1),\quad \vh=(0,1,1).
\end{split}
\end{equation}
The six \textbf{Kuhn simplices}~\cite{Kuhn} are
\begin{equation}\label{eq:kuhn-six-simplices}
\begin{split}
 &\Delta_1=\conv\{\va,\vb,\vc,\vg\},\quad
  \Delta_2=\conv\{\va,\vb,\vf,\vg\},\\
 &\Delta_3=\conv\{\va,\ve,\vf,\vg\},\quad
  \Delta_4=\conv\{\va,\ve,\vh,\vg\},\\
 &\Delta_5=\conv\{\va,\vd,\vh,\vg\},\quad
  \Delta_6=\conv\{\va,\vd,\vc,\vg\}.
\end{split}
\end{equation}

We first recall the corresponding planar construction.
\begin{proposition}[Two-dimensional triangulation principle]
\label{prop:two dimensional_tri}
For $\mathbf x,\mathbf y,\mathbf z,\mathbf u\in\mathbb R^2$,
\[
 \conv\{\mathbf x,\mathbf y,\mathbf z\}\subseteq
 \conv\{\mathbf u,\mathbf y,\mathbf z\}\cup
 \conv\{\mathbf x,\mathbf u,\mathbf z\}\cup
 \conv\{\mathbf x,\mathbf y,\mathbf u\}.
\]
\end{proposition}
\begin{proof}
If $\mathbf x,\mathbf y,\mathbf z$ are collinear, their convex hull is
one of their pairwise segments. Otherwise, for a point
$\mathbf p\ne\mathbf u$ in the triangle, extend the ray from $\mathbf u$
through $\mathbf p$ to its last point $\mathbf q$ in the triangle.
Then $\mathbf p\in[\mathbf u,\mathbf q]$, and $\mathbf q$ lies on an
edge. The triangle over that edge with vertex $\mathbf u$ contains
$\mathbf p$. The case $\mathbf p=\mathbf u$ is immediate.
\end{proof}

\begin{proposition}\label{lem:two dimensional_warmup}
Let $\va=(0,0)$, $\vb=(1,0)$, $\vc=(1,1)$, $\vd=(0,1)$, and put
$\Delta_1=\conv\{\va,\vb,\vc\}$ and
$\Delta_2=\conv\{\va,\vd,\vc\}$. For arbitrary
$\vx_1,\vx_2\in\mathbb R^2$, the convex body
\[
 K=\conv\bigl((\Delta_1+\vx_1)\cup(\Delta_2+\vx_2)\bigr)
\]
satisfies $K+\mathbb Z^2=\mathbb R^2$.
\end{proposition}
\begin{proof}
After translating $K$, we may assume that
$K=\conv(\Delta_1\cup(\Delta_2+\vx))$. The following triangles lie
in lattice translates of $K$:
\begin{align*}
 T_1&=\conv\{\va,\vd+\vx,\vd\}\subset K+(-1,0),\\
 T_2&=\conv\{\vd,\vd+\vx,\vc\}\subset K+(0,1),\\
 T_3&=\conv\{\vc,\vd+\vx,\va\}\subset K.
\end{align*}
Proposition~\ref{prop:two dimensional_tri}, with the additional point
$\vd+\vx$, gives $\Delta_2\subset T_1\cup T_2\cup T_3$.
Together with $\Delta_1\subset K$, this proves
$[0,1]^2\subset K+\mathbb Z^2$ and hence the assertion.
\end{proof}

\begin{remark}\label{rem:two dimensional_hexagon_tiling}
The planar construction is consistent with the classical description
of minimal lattice-covering bodies as centrally symmetric hexagons,
including parallelograms as degenerate cases; see
\cite{Dowker,Fary1950,Fejes1950}.
\end{remark}

\begin{proposition}[Three-dimensional triangulation principle]
\label{prop:three-dimensional_tri}
For $\mathbf x,\mathbf y,\mathbf z,\mathbf w,\mathbf u\in\mathbb R^3$,
\begin{align*}
 \conv\{\mathbf x,\mathbf y,\mathbf z,\mathbf w\}\subseteq{}&
 \conv\{\mathbf u,\mathbf y,\mathbf z,\mathbf w\}\cup
 \conv\{\mathbf x,\mathbf u,\mathbf z,\mathbf w\}\\
 &\cup\conv\{\mathbf x,\mathbf y,\mathbf u,\mathbf w\}\cup
 \conv\{\mathbf x,\mathbf y,\mathbf z,\mathbf u\}.
\end{align*}
\end{proposition}
\begin{proof}
If the four generators are affinely dependent, every point of their
convex hull belongs to the convex hull of at most three of them.
Otherwise the ray argument in
Proposition~\ref{prop:two dimensional_tri} applies, with a triangular
face in place of an edge.
\end{proof}

In the proof below, we first choose generic translation vectors, so that
every auxiliary tetrahedron is nondegenerate at the final parameter
$t=1$. Replacing each $\vx_i$ by $t\vx_i$ moves the vertices affinely
from the initial positions at $t=0$, where the nondegenerate pieces and
their lattice translates tile space, to this final configuration.
The convex hull of the translated Kuhn simplices varies during this
motion; containment in lattice translates of the prescribed $K$ is
needed only at $t=1$. The following lemma preserves coverage through
possible intermediate degeneracies. Arbitrary translation vectors will
then follow by approximation.

\begin{lemma}[Pairing triangular faces]
\label{lem:periodic-face-pairing}
For $1\le j\le N$, let
\[
 Q_j(t)=\conv\{\mathbf p_{j0}(t),\mathbf p_{j1}(t),
                      \mathbf p_{j2}(t),\mathbf p_{j3}(t)\},
 \qquad 0\le t\le1,
\]
where each vertex depends affinely on $t$. Assume that each $Q_j(t)$
is nondegenerate for at least one value of $t$. Suppose that the
$4N$ labelled triangular-face incidences can be paired. In each pair,
a fixed bijection of the three vertex labels identifies the two faces
after translation by one fixed vector of $\mathbb Z^3$, for every $t$.
If the nondegenerate members at $t=0$, together with their lattice
translates, cover almost every point exactly once, then
\[
             \bigcup_{j=1}^N(Q_j(t)+\mathbb Z^3)=\mathbb R^3
             \qquad(0\le t\le1).
\]
\end{lemma}
An elementary proof by counting covering tetrahedra is given in
Appendix~\ref{app:face-pairing}. Only the statement of the lemma
is needed below.

\begin{proof}[Proof of Theorem~\ref{thm:kuhn_cover}]
We construct 28 tetrahedra in lattice translates of $K$, pair their
112 triangular faces, and identify the subdivision obtained when all
six translations are zero.

Translate the configuration so that $\vx_1=\mathbf0$. For a lattice point
$\mathbf p$ and an index $0\le i\le5$, put
\begin{equation}\label{eq:theorem-one-colour-dictionary}
       \mathbf p^i=\mathbf p+\vx_{i+1}.
\end{equation}
The superscript records which translation is used. We retain the pair
$(\mathbf p,i)$ as a vertex label, so coincident points can still be
matched as the translations vary. In particular, $\mathbf p^0=\mathbf p$,
and $\mathbf p^i\in K$ whenever $\mathbf p$ is a vertex of $\Delta_{i+1}$.
In the following tables, brackets denote the convex hull of the
listed points; no ordering of the vertices is used.

\medskip
\noindent\textit{Tetrahedra associated with an octahedron.}
Put
\[
 \mathbf q_1=(2,1,0),\qquad \mathbf q_2=(2,1,1),\qquad
 \mathbf q_3=(2,2,1),\qquad \mathbf C=\vg^3,
\]
and consider the octahedron
\[
 \mathcal O=\conv\{\vb,\vg,\mathbf q_2,\mathbf q_1,\vc,\mathbf q_3\}.
\]
Four tetrahedra obtained from its alternating triangular faces are
\begin{equation}\label{eq:theorem-one-octahedral-cells}
\begin{array}{c|c|c}
 &\text{vertices}&\text{containing translate}\\ \hline
 O_1&[\mathbf C,\vb^0,\vc^0,\vg^0]&K\\
 O_2&[\mathbf C,\vb^0,\mathbf q_1^0,\mathbf q_2^0]&K+\vb\\
 O_3&[\mathbf C,\vc^0,\mathbf q_1^0,\mathbf q_3^0]&K+\vc\\
 O_4&[\mathbf C,\vg^0,\mathbf q_2^0,\mathbf q_3^0]&K+\vg.
\end{array}
\end{equation}
Every containment follows by subtracting the displayed lattice vector
and using \eqref{eq:kuhn-six-simplices}. For example,
\[
 O_2-\vb=[\vh^3,\va^0,\vc^0,\vg^0],\qquad
 O_3-\vc=[\ve^3,\va^0,\vb^0,\vg^0],\qquad
 O_4-\vg=[\va^3,\va^0,\vb^0,\vc^0].
\]
Each of the four points in these expressions belongs to $K$.

\medskip
\noindent\textit{Tetrahedra associated with $[\va,\ve,\vh,\vg]$.}
Write
\[
 \mathbf P=\ve^1,\quad \mathbf R=\ve^2,\quad
 \mathbf S=\vh^4,\quad \mathbf U=\vh^5.
\]
The next eleven tetrahedra also lie in the indicated translates of $K$:
\begin{equation}\label{eq:theorem-one-core-cells}
\begin{array}{c|c|c}
 &\text{vertices}&\text{containing translate}\\ \hline
 I_1 &[\mathbf P,\mathbf R,\mathbf S,\mathbf U]&K-\vb\\
 I_2 &[\va,\mathbf P,\mathbf R,\mathbf S]&K-\vb\\
 I_3 &[\vh,\mathbf P,\mathbf S,\mathbf U]&K-\vb\\
 I_4 &[\ve,\mathbf P,\mathbf R,\mathbf U]&K+\ve\\
 I_5 &[\vg,\mathbf R,\mathbf S,\mathbf U]&K+\ve\\
 I_6 &[\va,\vh,\mathbf P,\mathbf S]&K-\vb\\
 I_7 &[\va,\ve,\mathbf P,\mathbf R]&K-\vc\\
 I_8 &[\va,\vg,\mathbf R,\mathbf S]&K\\
 I_9 &[\ve,\vg,\mathbf R,\mathbf U]&K+\ve\\
 I_{10}&[\ve,\vh,\mathbf P,\mathbf U]&K+\ve-\vb\\
 I_{11}&[\vh,\vg,\mathbf S,\mathbf U]&K+\vh.
\end{array}
\end{equation}
Let
\[
 F_1=[\va,\ve,\vh],\quad
 F_2=[\va,\ve,\vg],\quad
 F_3=[\va,\vh,\vg],\quad
 F_4=[\ve,\vh,\vg],
\]
and put
\begin{equation}\label{eq:theorem-one-outer-apices}
 \boldsymbol\gamma_1=(-1,0,1)^3,\qquad
 \boldsymbol\gamma_2=\ve^3,\qquad
 \boldsymbol\gamma_3=\vh^3,\qquad
 \boldsymbol\gamma_4=(0,1,2)^3.
\end{equation}
Here the superscripts on the coordinate triples have the meaning in
\eqref{eq:theorem-one-colour-dictionary}. The last twelve tetrahedra are
\begin{equation}\label{eq:theorem-one-bridge-cells}
\begin{array}{c|c|c}
 &\text{vertices}&\text{containing translate}\\ \hline
 J_{11}&[\boldsymbol\gamma_1,\ve,\va,\mathbf P]&K-\vc\\
 J_{12}&[\boldsymbol\gamma_1,\vh,\va,\mathbf P]&K-\vb\\
 J_{13}&[\boldsymbol\gamma_1,\ve,\vh,\mathbf P]&K+\ve-\vb\\
 J_{21}&[\boldsymbol\gamma_2,\ve,\vg,\mathbf R]&K+\ve\\
 J_{22}&[\boldsymbol\gamma_2,\vg,\va,\mathbf R]&K\\
 J_{23}&[\boldsymbol\gamma_2,\va,\ve,\mathbf R]&K-\vc\\
 J_{31}&[\boldsymbol\gamma_3,\vg,\vh,\mathbf S]&K+\vh\\
 J_{32}&[\boldsymbol\gamma_3,\vh,\va,\mathbf S]&K-\vb\\
 J_{33}&[\boldsymbol\gamma_3,\va,\vg,\mathbf S]&K\\
 J_{41}&[\boldsymbol\gamma_4,\ve,\vg,\mathbf U]&K+\ve\\
 J_{42}&[\boldsymbol\gamma_4,\vg,\vh,\mathbf U]&K+\vh\\
 J_{43}&[\boldsymbol\gamma_4,\vh,\ve,\mathbf U]&K+\ve-\vb.
\end{array}
\end{equation}
The containments are checked as before; for instance,
\[
 I_7+\vc=[\vc^0,\vg^0,\vg^1,\vg^2],\qquad
 J_{41}-\ve=[\vh^3,\va^0,\vc^0,\vd^5].
\]
Together with $D=[\va^0,\vb^0,\vc^0,\vg^0]\subset K$, these are
28 tetrahedra, all contained in lattice translates of $K$. Denote
this family by $\mathcal Q$.

\medskip
\noindent\textit{Pairing the triangular faces.}
First consider $D,O_1,\ldots,O_4$. The faces opposite $\mathbf C$ in
$O_1,O_2,O_3,O_4$, translated by $\mathbf0,-\vb,-\vc,-\vg$, respectively,
are precisely the four faces of $D$. This gives four pairs.
The remaining twelve faces are triangles over the twelve edges of
$\mathcal O$ with vertex $\mathbf C$: each edge occurs in exactly one
of the four alternating base faces in
\eqref{eq:theorem-one-octahedral-cells}.
The other four alternating faces and the required lattice translations
are
\begin{equation}\label{eq:theorem-one-octahedral-face-pairing}
\begin{array}{c|c|c}
 &\text{face of }\mathcal O&\text{translation }\mathbf z_i\\ \hline
 F_1&[\mathbf q_1^0,\mathbf q_2^0,\mathbf q_3^0]&(2,1,0)\\
 F_2&[\vc^0,\vg^0,\mathbf q_3^0]&(1,1,0)\\
 F_3&[\vb^0,\vg^0,\mathbf q_2^0]&(1,0,0)\\
 F_4&[\vb^0,\vc^0,\mathbf q_1^0]&(1,0,-1).
\end{array}
\end{equation}
The translation $\mathbf z_i$ sends $F_i$ to the indicated face and
$\boldsymbol\gamma_i$ to $\mathbf C$. Thus the twelve remaining faces,
up to lattice translation, are
\begin{equation}\label{eq:kuhn-outer-triangles}
 [\boldsymbol\gamma_i,\mathbf u,\mathbf v],
 \qquad [\mathbf u,\mathbf v]\text{ an edge of }F_i,
 \qquad 1\le i\le4.
\end{equation}

Among $I_1,\ldots,I_{11}$ there are sixteen pairs of identical faces:
\begin{equation}\label{eq:kuhn-inner-face-pairs}
\begin{array}{c|c@{\qquad}c|c}
 \text{pair}&\text{common face}&\text{pair}&\text{common face}\\ \hline
 I_1,I_5&[\mathbf R,\mathbf S,\mathbf U]&I_3,I_{10}&[\vh,\mathbf P,\mathbf U]\\
 I_1,I_3&[\mathbf P,\mathbf S,\mathbf U]&I_3,I_6&[\vh,\mathbf P,\mathbf S]\\
 I_1,I_4&[\mathbf P,\mathbf R,\mathbf U]&I_4,I_9&[\ve,\mathbf R,\mathbf U]\\
 I_1,I_2&[\mathbf P,\mathbf R,\mathbf S]&I_4,I_{10}&[\ve,\mathbf P,\mathbf U]\\
 I_2,I_8&[\va,\mathbf R,\mathbf S]&I_4,I_7&[\ve,\mathbf P,\mathbf R]\\
 I_2,I_6&[\va,\mathbf P,\mathbf S]&I_5,I_{11}&[\vg,\mathbf S,\mathbf U]\\
 I_2,I_7&[\va,\mathbf P,\mathbf R]&I_5,I_9&[\vg,\mathbf R,\mathbf U]\\
 I_3,I_{11}&[\vh,\mathbf S,\mathbf U]&I_5,I_8&[\vg,\mathbf R,\mathbf S].
\end{array}
\end{equation}
Put $(\mathbf r_1,\mathbf r_2,\mathbf r_3,\mathbf r_4)
=(\mathbf P,\mathbf R,\mathbf S,\mathbf U)$.
Inspection of \eqref{eq:theorem-one-core-cells} shows that the twelve
unpaired faces are
\begin{equation}\label{eq:kuhn-inner-triangles}
 [\mathbf r_i,\mathbf u,\mathbf v],
 \qquad [\mathbf u,\mathbf v]\text{ an edge of }F_i,
 \qquad 1\le i\le4.
\end{equation}

For each $i$, the three tetrahedra $J_{i1},J_{i2},J_{i3}$ have the form
\[
 [\boldsymbol\gamma_i,\mathbf r_i,\mathbf u,\mathbf v],
 \qquad [\mathbf u,\mathbf v]\text{ an edge of }F_i.
\]
Their faces $[\boldsymbol\gamma_i,\mathbf r_i,\mathbf u]$ occur twice,
once for each of the two edges of $F_i$ incident with $\mathbf u$.
This gives three pairs for each $i$, hence twelve pairs in all.
The six faces left in this group are the three triangles in
\eqref{eq:kuhn-outer-triangles} and the three in
\eqref{eq:kuhn-inner-triangles}. They pair with the faces already
listed for $D,O_1,\ldots,O_4$ and $I_1,\ldots,I_{11}$.
We have therefore accounted for
\[
                     4+16+12+12+12=56
\]
pairs, containing all $28\cdot4=112$ triangular faces. These are
identities between vertex labels, so every pairing persists under
arbitrary choices of $\vx_2,\ldots,\vx_6$.

\medskip
\noindent\textit{The subdivision at zero translations.}
When $\vx_1=\cdots=\vx_6=\mathbf0$, exactly six of the 28 tetrahedra are
nondegenerate. Choosing lattice translates inside the unit cube gives
\begin{equation}\label{eq:theorem-one-initial-subdivision}
\begin{array}{c|c@{\qquad}c|c}
 \text{tetrahedron}&\text{vertices}&\text{tetrahedron}&\text{vertices}\\ \hline
 D&[\va,\vb,\vc,\vg]&I_8&[\va,\ve,\vg,\vh]\\
 O_2-\vb&[\va,\vc,\vg,\vh]&J_{12}+\vb&[\vb,\ve,\vf,\vg]\\
 O_3-\vc&[\va,\vb,\ve,\vg]&J_{41}-\ve&[\va,\vc,\vd,\vh].
\end{array}
\end{equation}
These six tetrahedra form a subdivision of the cube. To see this
directly, start with \eqref{eq:kuhn-six-simplices} and keep $\Delta_1$
and $\Delta_4$. The union $\Delta_2\cup\Delta_3$ is the square pyramid
\[
        \conv\{\va,\vb,\ve,\vf,\vg\},
\]
with base $[\va,\vb,\vf,\ve]$ in the plane $y=0$ and vertex $\vg$.
Replacing its base diagonal $[\va,\vf]$ by $[\vb,\ve]$ gives the two
tetrahedra $[\va,\vb,\ve,\vg]$ and $[\vb,\ve,\vf,\vg]$.
Similarly, $\Delta_5\cup\Delta_6$ is the square pyramid
\[
        \conv\{\va,\vc,\vd,\vg,\vh\},
\]
with base $[\vd,\vc,\vg,\vh]$ in the plane $y=1$ and vertex $\va$.
Replacing $[\vd,\vg]$ by $[\vc,\vh]$ gives
$[\va,\vc,\vd,\vh]$ and $[\va,\vc,\vg,\vh]$.
Figure~\ref{fig:kuhn-diagonal-changes} shows the two changes.
Each changes only the diagonal in the base of one square pyramid,
so the six resulting tetrahedra still cover the cube exactly once
away from their faces. The changes also agree on opposite cube faces,
since $[\vb,\ve]+(0,1,0)=[\vc,\vh]$.

\begin{figure}[tbp]
\centering
\begin{tikzpicture}[scale=0.92,line cap=round,line join=round]
 \begin{scope}
  \coordinate (a) at (0,0); \coordinate (b) at (2.5,0);
  \coordinate (f) at (2.5,2.5); \coordinate (e) at (0,2.5);
  \fill[gray!9] (a)--(b)--(e)--cycle;
  \draw (a)--(b)--(f)--(e)--cycle;
  \draw[densely dashed,gray] (a)--(f);
  \draw[thick] (b)--(e);
  \node[below left] at (a) {$\va$}; \node[below right] at (b) {$\vb$};
  \node[above right] at (f) {$\vf$}; \node[above left] at (e) {$\ve$};
  \node at (1.25,3.1) {vertex $\vg$, base $y=0$};
  \node at (1.25,-0.8) {$[\va,\vf]\longmapsto[\vb,\ve]$};
 \end{scope}
 \begin{scope}[xshift=6cm]
  \coordinate (d) at (0,0); \coordinate (c) at (2.5,0);
  \coordinate (g) at (2.5,2.5); \coordinate (h) at (0,2.5);
  \fill[gray!9] (d)--(c)--(h)--cycle;
  \draw (d)--(c)--(g)--(h)--cycle;
  \draw[densely dashed,gray] (d)--(g);
  \draw[thick] (c)--(h);
  \node[below left] at (d) {$\vd$}; \node[below right] at (c) {$\vc$};
  \node[above right] at (g) {$\vg$}; \node[above left] at (h) {$\vh$};
  \node at (1.25,3.1) {vertex $\va$, base $y=1$};
  \node at (1.25,-0.8) {$[\vd,\vg]\longmapsto[\vc,\vh]$};
 \end{scope}
\end{tikzpicture}
\caption{The bases of the two square pyramids. Joining each base
triangle to the indicated vertex gives a tetrahedron. Replacing the
dashed diagonals by the solid diagonals gives the six tetrahedra in
\eqref{eq:theorem-one-initial-subdivision}.}
\label{fig:kuhn-diagonal-changes}
\end{figure}

\medskip
\noindent\textit{Independent translations.}
For a generic choice of $\vx_2,\ldots,\vx_6$, all 28 tetrahedra are
nondegenerate. Indeed, except for $D$, the vertices in each row consist
of at most three affinely independent points of label 0 and at most
one point of each other label. The latter points can be moved
independently, so no row determinant is the zero polynomial in the
translation coordinates. The complement of these finitely many zero
sets is dense.

For such a choice, replace every $\vx_i$ by $t\vx_i$.
The face pairings persist, and the initial subdivision is
\eqref{eq:theorem-one-initial-subdivision}. Lemma~\ref{lem:periodic-face-pairing}
therefore gives
\[
 \mathbb R^3=\bigcup_{Q\in\mathcal Q}
                 (Q+\mathbb Z^3)
 \subseteq K+\mathbb Z^3.
\]
For arbitrary translation vectors,
approximate them by generic choices. The corresponding convex hulls
converge to $K$ and remain uniformly bounded. For each point of space,
only finitely many lattice translates of these convex hulls can
contain it. Passing to a subsequence with one fixed lattice translate
proves coverage in the limit.
\end{proof}

\begin{lemma}[Balitskiy~{\cite[Lemma~5.3]{Balitskiy}}]
\label{lem:balitskiy}
Let $\mathcal T$ be a locally finite triangulation of $\mathbb R^n$
whose simplices have uniformly bounded diameters, and let $K$ be a
convex body. If each maximal simplex of $\mathcal T$ is contained in
some translate of $K$, then every translate of $K$ contains a vertex
of $\mathcal T$.
\end{lemma}

\begin{remark}\label{rem:kuhn-all-dimensions}
Balitskiy's lemma also gives the covering assertion in every dimension.
For $\sigma\in S_n$, let
\[
 \Delta_\sigma=\conv\{\mathbf0,\mathbf e_{\sigma(1)},
     \mathbf e_{\sigma(1)}+\mathbf e_{\sigma(2)},\ldots,
     \mathbf e_{\sigma(1)}+\cdots+\mathbf e_{\sigma(n)}\},
\]
where $\mathbf e_1,\ldots,\mathbf e_n$ are the standard basis vectors
and $S_n$ is the permutation group. These simplices and their lattice
translates form the standard Kuhn triangulation, with vertex set
$\mathbb Z^n$. If $K$ contains a translate of each $\Delta_\sigma$,
apply Lemma~\ref{lem:balitskiy} to $K-\mathbf y$ for any
$\mathbf y\in\mathbb R^n$. It gives a point $\mathbf k\in K$ with
$\mathbf k-\mathbf y\in\mathbb Z^n$, and hence
$\mathbf y\in K+\mathbb Z^n$.
The three-dimensional construction above makes the auxiliary tetrahedra
and their persistent face pairings explicit.
\end{remark}

\section{Definition and Properties of Minimal Covering Bodies}\label{sec:definition}

\subsection{Definitions and Finiteness}
Motivated by the Kuhn-simplex constructions, we now study inclusion-minimal
lattice covering bodies.  The same notion and the polytopality theorem
below were obtained independently by Codenotti, Freyer, and
Krivoku\'ca~\cite{Codenotti}.

\begin{definition}[Minimal covering body]
Let $K\subset\Rr^n$ be a convex body and let $\Lambda$ be a full-rank
lattice such that $K+\Lambda=\Rr^n$.  If no proper convex subbody
$K'\subsetneq K$ satisfies $K'+\Lambda=\Rr^n$, then $K$ is a
\textbf{minimal covering body} with respect to $\Lambda$.
We also call $K+\Lambda$ a minimal covering of $\Rr^n$.
\end{definition}

The comparison is with all convex subbodies, including non-symmetric
ones when $K$ is centrally symmetric.  This is minimality under
inclusion, rather than minimum volume or minimum covering density.
We begin with the precise truncation fact
needed in the proof of polytopality.  A point $\vv\in K$ is exposed if some
linear functional attains its maximum on $K$ only at $\vv$.

The following lemma also follows from
\cite[Lemma~2.2(iii), equation~(2.2)]{Codenotti} after a linear
change of lattice coordinates. We include a direct proof by truncation.

\begin{lemma}[Exposed caps]\label{lem:exposed-cap}
Let $K$ be a minimal covering body with respect to $\Lambda$.
If $\vv$ is an exposed point of $K$, then every open neighborhood $U$ of
$\vv$ contains a point $\mathbf{p}\in K$ such that
\[
                 \mathbf{p}\notin K-\vz\qquad(\vnull\ne \vz\in\Lambda).
\]
\end{lemma}
\begin{proof}
Choose $r>0$ such that $B(\vv,r)\subset U$ and
$2r<\min\{\|\vz\|:\vnull\ne \vz\in\Lambda\}$, where $B(\vv,r)$ is the open
Euclidean ball.  Since $\vv$ is exposed, there is a linear functional $f$
with $h:=f(\vv)=\max_K f$ and $K\cap\{f=h\}=\{\vv\}$.  For sufficiently
small $\delta>0$, compactness gives
\[
 C_\delta:=K\cap\{f\ge h-\delta\}\subset B(\vv,r),
 \qquad K':=K\cap\{f\le h-\delta\}\subsetneq K.
\]
The inclusion follows because a sequence outside $B(\vv,r)$ with
$f$-values tending to $h$ would have a limit different from $\vv$ in
$K\cap\{f=h\}$. Choose an interior point $\mathbf{q}$ of $K$ and also require
$\delta<h-f(\mathbf{q})$. Then $\mathbf{q}\in\inte K'$, so $K'$ has nonempty interior.

Suppose that every point of $K\cap U$ belongs to some $K-\vz$ with
$\vnull\ne \vz\in\Lambda$.  For $\mathbf{p}\in K\setminus K'$, choose such a $\vz$;
then $\mathbf{p}+\vz\in K$.  The points $\mathbf{p}$ and $\mathbf{p}+\vz$ cannot both lie in
$C_\delta$, since its diameter is smaller than $\|\vz\|$; see
Figure~\ref{fig:cap-and-ray}(a).  Hence
$\mathbf{p}+\vz\in K'$, and therefore
\[
 K\setminus K'\subset K'+\Lambda,\qquad
 \Rr^n=K+\Lambda
       =\bigl(K'\cup(K\setminus K')\bigr)+\Lambda
       =K'+\Lambda.
\]
This contradicts minimality.
\end{proof}

Here a $\Lambda$-lattice polytope is the convex hull of finitely many
points of $\Lambda$; it may have dimension less than $n$.
For a convex body $K$, a full-rank lattice $\Lambda$, and a point
$\vc\in K$, define the \textbf{contact set} at $\vc$ by
\begin{equation}\label{eq:contact-set-definition}
 \mathcal Z_K(\vc)=(K-\vc)\cap\Lambda
              =\{\vz\in\Lambda:\vc+\vz\in K\}.
\end{equation}
Thus $\vz\in\mathcal Z_K(\vc)$ means that the translate $K-\vz$ contains
$\vc$.  The lattice is fixed when it is suppressed from the notation.
The associated \textbf{root contact polytope} is
\begin{equation}\label{eq:root-contact-polytope}
                        P_{\vc}=\conv\mathcal Z_K(\vc),
                        \qquad\text{with distinguished point }\vnull.
\end{equation}
These data are determined by $(K,\Lambda,\vc)$ and involve no choice of
generating polytopes. Compactness of $K$ and discreteness of $\Lambda$
make $\mathcal Z_K(\vc)$ finite, so $P_{\vc}$ is a lattice polytope,
possibly of lower dimension.

\begin{proposition}[Polytopality]\label{prop:finiteness}
Every minimal covering body $K$ with respect to a lattice $\Lambda$ is
a polytope.  Moreover, $K$ is the convex hull of translations of
$\Lambda$-lattice polytopes.
\end{proposition}
\begin{proof}
We show that different exposed points have different contact sets.
Since
\begin{equation}\label{eq:finite-contact-universe}
                \mathcal Z_K(\vc)\subset (K-K)\cap\Lambda
                \qquad(\vc\in K),
\end{equation}
and the right-hand side is finite, this will bound the number of
exposed points.

Suppose that distinct exposed points $\vv,\vw$ have the same contact set
$S$.  Choose a nonzero vector $\va$ such that
\[
 f(\vx):=\langle \va,\vx\rangle\le h:=f(\vv)\quad(\vx\in K),
 \qquad K\cap H=\{\vv\},\quad H:=\{f=h\}.
\]
In particular, $f(\vw)<h$.  Only finitely many lattice translates of
$K$ meet $B(\vv,1)$.  Each of these translates that does not contain $\vv$
has positive distance from $\vv$.  Consequently, there is $r>0$ such
that
\begin{equation}\label{eq:nearby-contact-translates}
 (K-\vz)\cap B(\vv,r)\ne\varnothing\ \Longrightarrow\ \vv\in K-\vz
 \qquad(\vz\in\Lambda).
\end{equation}

Choose $\tau>0$ with $\tau\|\va\|<r$.  Since $\vv+\tau\va\notin K$,
the covering property gives a nonzero $\vz_0\in\Lambda$ such that
$\vv+\tau\va\in K-\vz_0$.  By \eqref{eq:nearby-contact-translates},
$\vz_0\in S$, so both $\vv$ and $\vw$ belong to $K-\vz_0$.

For $\mathbf{p}\in K$ sufficiently close to $\vv$, we have $f(\mathbf{p})>f(\vw)$.
The ray from $\vw$ through $\mathbf{p}$ meets $H$ at the point shown in
Figure~\ref{fig:cap-and-ray}(b), namely
\[
 \mathbf{q}_{\mathbf{p}}=\vw+\frac{h-f(\vw)}{f(\mathbf{p})-f(\vw)}(\mathbf{p}-\vw).
\]
The denominator is positive, and the coefficient is at least $1$.
Thus $\mathbf{p}\in[\vw,\mathbf{q}_{\mathbf{p}}]$.  As $\mathbf{p}\to \vv$, we have $\mathbf{q}_{\mathbf{p}}\to \vv$; hence one
open neighborhood $U$ of $\vv$ can be chosen so that $f(\mathbf{p})>f(\vw)$ and
$\mathbf{q}_{\mathbf{p}}\in B(\vv,r)$ for every $\mathbf{p}\in K\cap U$.  We consider the two
possible positions of $\mathbf{q}_{\mathbf{p}}$.

If $\mathbf{q}_{\mathbf{p}}\ne \vv$, then $\mathbf{q}_{\mathbf{p}}\notin K$, because $K\cap H=\{\vv\}$.
Choose a translate $K-\vz$ containing $\mathbf{q}_{\mathbf{p}}$.  Here $\vz\ne\vnull$, and
\eqref{eq:nearby-contact-translates} gives $\vz\in S$.  Since $S$ is
also the contact set at $\vw$, we have $\vw\in K-\vz$.  Convexity now gives
$\mathbf{p}\in[\vw,\mathbf{q}_{\mathbf{p}}]\subset K-\vz$.

If $\mathbf{q}_{\mathbf{p}}=\vv$, use the fixed translate $K-\vz_0$: it contains $[\vw,\vv]$,
and hence contains $\mathbf{p}$.

In both cases, every $\mathbf{p}\in K\cap U$ belongs to a translate $K-\vz$
with $\vz\ne\vnull$.  This contradicts Lemma~\ref{lem:exposed-cap}.  Thus exposed
points have distinct contact sets and, by
\eqref{eq:finite-contact-universe}, there are only finitely many of
them.  By Straszewicz's theorem, $K$ is the closed convex hull of its
exposed points~\cite[Theorem~1.4.7]{Schneider}; therefore $K$ is a
polytope.

Let $V(K)$ denote its vertex set.  For every $\vv\in V(K)$,
$\vv+P_{\vv}\subset K$ and $\vv\in \vv+P_{\vv}$, so
\[
                 K=\conv\bigcup_{\vv\in V(K)}(\vv+P_{\vv}).
\]
This proves the final assertion.
\end{proof}

\begin{figure}[htbp]
\centering
\begin{tikzpicture}[scale=1.18, every node/.style={font=\small}]
 \begin{scope}
  \path[fill=gray!8] (0,0.5)--(0.8,-1)--(2,-0.6)--(3.4,0)
      --(2,1.1)--(0.4,1.5)--cycle;
  \begin{scope}
   \clip (0,0.5)--(0.8,-1)--(2,-0.6)--(3.4,0)
      --(2,1.1)--(0.4,1.5)--cycle;
   \fill[gray!40] (2.8,-1.3) rectangle (3.5,1.7);
  \end{scope}
  \draw[thick] (0,0.5)--(0.8,-1)--(2,-0.6)--(3.4,0)
      --(2,1.1)--(0.4,1.5)--cycle;
  \draw[densely dashed] (3.4,0) circle (0.82);
  \draw (2.8,-1.22)--(2.8,1.52);
  \draw (3.4,-1.22)--(3.4,1.52);
  \node[above] at (3.4,1.52) {$f=h$};
  \node[below,align=center] at (2.8,-1.24) {$f=h-\delta$};
  \node at (1.1,0.8) {$K'$};
  \node[above right] at (3.0,0.2) {$C_\delta$};
  \fill (3.4,0) circle (1.5pt);
  \node[right] at (3.4,0) {$\vv$};
  \fill (3.05,0.03) circle (1.5pt);
  \node[below] at (3.05,0.03) {$\mathbf p$};
  \fill (1.3,0.03) circle (1.5pt);
  \node[below] at (1.3,0.03) {$\mathbf p+\vz$};
  \draw[->,thick] (2.96,0.03)--(1.38,0.03)
      node[midway,above] {$\vz$};
  \node[align=center] at (1.8,-1.9) {(a) Removing a small cap};
 \end{scope}
 \begin{scope}[xshift=5.2cm]
  \fill[gray!8] (0,0)--(1,-1.1)--(2.6,-0.6)--(3.4,0.9)
      --(1.3,1.45)--cycle;
  \fill[gray!35] (0,0)--(3.4,0.9)--(3.4,0.388571)--cycle;
  \draw[thick] (0,0)--(1,-1.1)--(2.6,-0.6)--(3.4,0.9)
      --(1.3,1.45)--cycle;
  \draw (3.4,-1.22)--(3.4,1.55);
  \node[above] at (3.4,1.55) {$H$};
  \draw[densely dashed] (0,0)--(3.4,0.9);
  \draw[->,thick] (0,0)--(3.72,0.425143);
  \fill (0,0) circle (1.5pt);
  \node[left] at (0,0) {$\vw$};
  \fill (3.4,0.9) circle (1.5pt);
  \node[right] at (3.4,0.9) {$\vv$};
  \fill (2.8,0.32) circle (1.5pt);
  \node[below] at (2.8,0.32) {$\mathbf p$};
  \fill (3.4,0.388571) circle (1.5pt);
  \node[below right] at (3.4,0.388571) {$\mathbf q_{\mathbf p}$};
  \node at (1.3,-0.5) {$K$};
  \node[align=center] at (1.8,-1.9) {(b) Equal contact sets at $\vv,\vw$};
 \end{scope}
\end{tikzpicture}
\caption{The two local steps in the polytopality proof, shown
schematically in a plane. In (a), the cap lies in a ball whose diameter
is smaller than the length of every nonzero lattice vector; an alternative
representative $\mathbf p+\vz\in K$ therefore belongs to $K'$.
In (b), the ray from $\vw$ through $\mathbf p$ meets the exposing
hyperplane at $\mathbf q_{\mathbf p}$.
If $\mathbf q_{\mathbf p}\ne\vv$, the shaded triangle
$\conv\{\vw,\vv,\mathbf q_{\mathbf p}\}$ lies in a translate
$K-\vz$ containing both endpoints of the ray segment.
If $\mathbf q_{\mathbf p}=\vv$, the dashed segment is contained in
the fixed translate $K-\vz_0$.}
\label{fig:cap-and-ray}
\end{figure}

The contact-set separation used here also follows from
\cite[Lemma~2.3]{Codenotti}.

The contact polytopes also suggest a boundary restriction.  If $\vv$
is a vertex of $K$, Lemma~\ref{lem:exposed-cap} gives
\[
                         K\cap(\vv+\Lambda)\subset\bd K:
\]
indeed, $\vv+\vz\in\inte K$ for a nonzero $\vz\in\Lambda$ would
put a whole neighborhood of $\vv$ in $K-\vz$.  Consequently, if a
translate $\mathbf t+Q\subset K$ of a lattice polytope has a vertex
that is also a vertex of $K$, all vertices of $\mathbf t+Q$ lie in
$\bd K$, since they are in the same lattice coset.  We ask whether
this conclusion holds without the common-vertex assumption.
\begin{conjecture}\label{conj:embedded-lattice-polytope-boundary}
Let $K$ be a minimal covering body with respect to $\Lambda$, and let
$Q$ be a full-dimensional $\Lambda$-lattice polytope.  If $\mathbf{t}+Q\subset K$
for a vector $\mathbf{t}\in\Rr^n$, then $\mathbf{t}+V(Q)\subset\bd K$.
\end{conjecture}

An affirmative answer would restrict the placement of a full-dimensional
lattice polytope in $K$ to positions with all vertices on boundary
faces.  For an embedded lattice simplex, for example, assigning its
vertices to faces of $K$ would give supporting-hyperplane equalities
that its placement must satisfy.  The results below use only the
common-vertex case proved above.

We next make precise the arithmetic information used to distinguish
our examples.  Write $\mathcal F(K)$ for the face lattice of $K$,
including $\varnothing$ and $K$.  For $\mathbf{p}\in K$, its
\textbf{face carrier} $F_K(\mathbf{p})$ is the unique face of $K$ whose relative
interior contains $\mathbf{p}$.  For $\vv\in V(K)$, define
\[
 \iota_{\vv}:\mathcal Z_K(\vv)\longrightarrow\mathcal F(K),
 \qquad \iota_{\vv}(\vz)=F_K(\vv+\vz).
\]
This map records the face containing each contact point, in addition to
its lattice displacement from $\vv$.

\begin{definition}[Arithmetic contact type]\label{def:type}
The \textbf{arithmetic contact profile} of a minimal covering pair
$(K,\Lambda)$ consists of its face lattice $\mathcal F(K)$, the lattice
$\Lambda$, and the rooted contact data
$\{(\mathcal Z_K(\vv),\vnull,\iota_{\vv}):\vv\in V(K)\}$.

For another minimal covering pair $(K',\Lambda')$, a
\textbf{profile isomorphism} consists of a lattice isomorphism
$U:\Lambda\to\Lambda'$, a vertex bijection
$\phi:V(K)\to V(K')$, and a face-lattice isomorphism
$\Phi:\mathcal F(K)\to\mathcal F(K')$, satisfying
\begin{align}
 \Phi(\{\vv\})&=\{\phi(\vv)\},
       &&\vv\in V(K),\label{eq:contact-vertex-equivalence}\\
 U\mathcal Z_K(\vv)&=\mathcal Z_{K'}(\phi(\vv)),
       &&\vv\in V(K),\label{eq:contact-type-equivalence}\\
 \Phi\bigl(\iota_{\vv}(\vz)\bigr)&=\iota'_{\phi(\vv)}(U\vz),
       &&\vv\in V(K),\ \vz\in\mathcal Z_K(\vv).
       \label{eq:contact-carrier-equivalence}
\end{align}
The first identity matches the singleton faces with the vertex bijection.
Here $\iota'$ denotes the face-carrier maps for $K'$.
The pairs have the same \textbf{arithmetic contact type} if such an
isomorphism exists.
\end{definition}

Identity maps, inverses, and composition show that this defines an
equivalence relation.  The map $U$ is the same at every vertex, and we
also use $U$ for its unique linear extension to $\Rr^n$.
Ordinary combinatorial type records only the face lattice; arithmetic
contact type additionally records lattice contacts and the faces on
which they occur.  In particular, changing an arbitrary family of
translated polytopes used to generate $K$ does not change this type.

An affine map $\vx\mapsto A\vx+\mathbf{t}$ with $A\Lambda=\Lambda'$ induces a
profile isomorphism, since
\[
 \mathcal Z_{AK+\mathbf{t}}(A\vv+\mathbf{t})=A\mathcal Z_K(\vv),\qquad
 F_{AK+\mathbf{t}}(A\mathbf{p}+\mathbf{t})=AF_K(\mathbf{p})+\mathbf{t}.
\]
Conversely, suppose that all vertices of $K$ lie in one coset of
$\Lambda$.  Then, for every $\vv\in V(K)$,
\begin{equation}\label{eq:single-coset-contact-polytope}
                              P_{\vv}=K-\vv.
\end{equation}
Indeed, each vertex $\vw$ has $\vw-\vv\in\mathcal Z_K(\vv)$, while
$P_{\vv}\subset K-\vv$.  If a profile isomorphism from $(K,\Lambda)$ to
$(K',\Lambda')$ is given, fix a vertex $\vv_*\in V(K)$.  Applying
\eqref{eq:contact-carrier-equivalence} to $\vz=\vw-\vv_*$ gives
\[
                   \phi(\vw)=\phi(\vv_*)+U(\vw-\vv_*)
                   \qquad(\vw\in V(K)).
\]
Thus $\vx\mapsto U(\vx-\vv_*)+\phi(\vv_*)$ maps $K$ onto $K'$.  This explains
why the type determines the lattice-affine equivalence class for the
families below, whose vertices lie in one lattice coset.  When the
lattice is fixed, $|\det U|=1$; hence the ordinary volumes of
corresponding root contact polytopes are equal.

We distinguish this intrinsic type from the Kuhn construction in
Theorem~\ref{thm:kuhn_cover}.  Recall the Kuhn simplices
$\Delta_\sigma(B)$ associated with an ordered lattice basis $B$.
A \textbf{complete translated Kuhn family} consists of one translate
$\mathbf{t}_\sigma+\Delta_\sigma(B)$ for each $\sigma\in S_3$, with arbitrary
$\mathbf{t}_\sigma\in\Rr^3$.  If $K$ contains such a family, then
Theorem~\ref{thm:kuhn_cover} gives $K+\Lambda=\Rr^3$.  Its six
simplices need not meet face to face inside $K$.

\subsection{Boundary Constraints}

\begin{definition}[Supporting cone]\label{def:cone}
Let $K$ be a convex body and $\vx\in\bd K$.  Its \textbf{supporting cone}
(or tangent cone) at $\vx$ is
\[
 \mathrm{cone}_K(\vx)=\operatorname{cl}
       \{\alpha(\vy-\vx):\vy\in K,\ \alpha\ge0\}.
\]
Thus the cone consists of the directions from $\vx$ into $K$ and their
limits.
\end{definition}

Every vertex of a polytope is exposed.  Hence Lemma~\ref{lem:exposed-cap}
implies that, for a vertex $\vx$ of a minimal covering body,
$\vx+\vz\notin\inte K$ whenever $\vnull\ne \vz\in\Lambda$.  If the contact is
on the boundary, we have the following stronger restriction.

The following condition also follows from
\cite[Corollary~2.5, equation~(2.7)]{Codenotti}, which excludes
containment in the union of the cones at the other lattice contacts.
We give a direct proof for one contact.

\begin{lemma}[Vertex condition]\label{lem:vertex condition}
Let $K$ be a minimal covering body with respect to $\Lambda$.  If $\vx$
is a vertex of $K$ and $\vx+\vz\in\bd K$ for
$\vz\in\Lambda\setminus\{\vnull\}$, then
\[
                    \mathrm{cone}_K(\vx)
                    \not\subseteq\mathrm{cone}_K(\vx+\vz).
\]
\end{lemma}
\begin{proof}
Suppose that the inclusion holds.  By Proposition~\ref{prop:finiteness},
$K$ is a polytope. Choose $\varepsilon>0$ so small that $K$ agrees
with its translated supporting cone in the $\varepsilon$-neighborhoods
of both $\vx$ and $\vx+\vz$. Then
\[
 K\cap B(\vx,\varepsilon)
  =(\vx+\mathrm{cone}_K(\vx))\cap B(\vx,\varepsilon)
  \subset (K-\vz)\cap B(\vx,\varepsilon).
\]
We may also require $2\varepsilon<\|\vz\|$.  Truncate the vertex $\vx$
by a hyperplane so close to it that the remaining full-dimensional
polytope $K'\subsetneq K$ satisfies
$K\setminus K'\subset B(\vx,\varepsilon)$.  For every removed point $\mathbf{p}$,
the displayed inclusion gives $\mathbf{p}+\vz\in K$.  Since
$\mathbf{p}+\vz\in B(\vx+\vz,\varepsilon)$ and this ball is disjoint from
$B(\vx,\varepsilon)$, the point $\mathbf{p}+\vz$ was not removed.  Thus
$K\setminus K'\subset K'-\vz$, and
\[
 \Rr^n=K+\Lambda
       =\bigl(K'\cup(K\setminus K')\bigr)+\Lambda
       =(K'+\Lambda)\cup\bigl((K\setminus K')+\Lambda\bigr)
       =K'+\Lambda.
\]
This contradicts minimality.
\end{proof}

\begin{proposition}[Boundary lattice segments]
\label{prop:no-long-lattice-segment}
Let $K$ be a minimal covering body with respect to a lattice
$\Lambda\subset\Rr^3$.  If $\vx\in\Rr^3$,
$\vz\in\Lambda\setminus\{\vnull\}$, and $t>1$, then
$[\vx,\vx+t\vz]$ is not contained in $\bd K$.
\end{proposition}
\begin{proof}
Suppose that the segment is contained in $\bd K$.  A supporting plane
at its midpoint contains the whole segment.  The face of $K$ in this
plane is contained in a facet $F$, which is a convex polygon by
Proposition~\ref{prop:finiteness}.

Choose affine coordinates $(s,r)$ in the plane of $F$ so that an
increase of $1$ in $s$ corresponds to translation by $\vz$.  The
projection of $F$ onto the $r$-axis is a closed interval $I$.  For
a fixed $r\in I$, its nonempty horizontal section is
\[
 F\cap\{(s,r):s\in\mathbb R\}
     =\{(s,r):a(r)\le s\le b(r)\}.
\]
Here $a$ is convex and piecewise affine, $b$ is concave and piecewise
affine, and $L(r):=b(r)-a(r)$ is the length of the section measured in
units of $\vz$.  Its maximum $L_{\max}$ is at least $t>1$.

A longest such section can be chosen with a vertex of $F$ as an
endpoint.  To see this, if the maximum occurs at an endpoint of $I$,
that section is an edge or a vertex of $F$.  Otherwise the set of
maximizers of the concave function $L$ is a closed interval inside
$I$.  Choose an endpoint $r_0$ of this interval, allowing an isolated
maximum.  If neither $a$ nor $b$ changed its affine formula at $r_0$,
then $L$ would be affine on a neighborhood of $r_0$.  A local maximum
would force that affine function to be constant, extending the
maximizer interval past $r_0$, a contradiction.  Thus at least one
section endpoint is a vertex of $F$.

After reversing $\vz$ if necessary, the chosen chord has the form
$[\vv,\vv+L_{\max}\vz]$, where $\vv$ is a vertex of $F$ and hence of $K$.
The point $\vv+\vz$ lies in its relative interior.  Therefore every
supporting hyperplane of $K$ at $\vv+\vz$ contains the direction $\vz$ and
also contains $\vv$.  The corresponding supporting half-spaces give
\[
                   \mathrm{cone}_K(\vv)
                   \subseteq\mathrm{cone}_K(\vv+\vz),
\]
contrary to Lemma~\ref{lem:vertex condition}.
\end{proof}

The reduction to a vertex above uses that a facet in three dimensions
is a polygon. Once the segment starts at a vertex, the
supporting-hyperplane argument applies in every dimension and gives
the following consequence of Lemma~\ref{lem:vertex condition}.
\begin{proposition}\label{prop:vertex-ended-lattice-segment}
Let $K$ be a minimal covering body with respect to a lattice
$\Lambda\subset\Rr^n$.  If $\vv$ is a vertex of $K$,
$\vz\in\Lambda\setminus\{\vnull\}$, and $t>1$, then
$[\vv,\vv+t\vz]$ is not contained in $\bd K$.
\end{proposition}

\subsection{The Symmetric Case in Dimension Three}

In the plane, every lattice covering body contains a centrally symmetric
hexagon or a parallelogram that tiles by the same lattice.  In three
dimensions, Fedorov's five types are the parallelotope, hexagonal prism,
rhombic dodecahedron, elongated dodecahedron, and truncated
octahedron~\cite{Federov}.  The following proposition places all five
within the construction of Theorem~\ref{thm:kuhn_cover}.

For a parallelohedron $P$, write $\Lambda(P)$ for the translation
lattice of its face-to-face tiling: the intersection of any two tiles
is a face of both.  This tiling is determined by $P$ up to an overall
translation~\cite[Section~1]{GarberGavrilyukMagazinov2015}.

\begin{proposition}[Kuhn representation of three-dimensional parallelohedra]
\label{prop:fedorov-kuhn}\label{rem:kuhn_reduced}
Let $P\subset\Rr^3$ be a parallelohedron.  There are an ordered basis
$B$ of $\Lambda(P)$ and vectors $\mathbf t_\sigma\in\Rr^3$,
$\sigma\in S_3$, such that
\begin{equation}\label{eq:fedorov-kuhn-hull}
       P=\conv\bigcup_{\sigma\in S_3}
                    \bigl(\mathbf t_\sigma+\Delta_\sigma(B)\bigr).
\end{equation}
\end{proposition}
\begin{proof}
We first prove the representation for a lattice Voronoi cell, then
transfer it to $P$ by affine equivalence. A suitable lattice basis makes
the circumsphere of each Kuhn simplex empty of lattice points in its
interior.  Translating each simplex by the negative of its circumcenter
then places it in the Voronoi cell.

Let $\Gamma\subset\Rr^3$ be a lattice.  Its Voronoi cell at the origin is
\[
 V(\Gamma)=\{\vx\in\Rr^3:\|\vx\|\le\|\vx-\boldsymbol\gamma\|
                       \text{ for all }\boldsymbol\gamma\in\Gamma\}.
\]
A \textbf{superbase} consists
of a lattice basis $\vb_1,\vb_2,\vb_3$ and the vector
$\vb_0=-\vb_1-\vb_2-\vb_3$.  It is \textbf{obtuse} if
$\vb_i\cdot\vb_j\le0$ for all distinct $i,j\in\{0,1,2,3\}$.
Every three-dimensional lattice has such a superbase
\cite[Theorem~8]{ConwaySloane1992}.  Here is the short reduction
argument.  Among all superbases choose one minimizing
$E=\sum_{i=0}^3\|\vb_i\|^2$; the minimum exists because there are
only finitely many choices below any fixed bound for $E$.
If $\vb_i\cdot\vb_j>0$, let $k,l$ be the other two indices and replace
\[
 \vb_i,\vb_j,\vb_k,\vb_l
 \quad\text{by}\quad
 -\vb_i,\vb_j,\vb_k+\vb_i,\vb_l+\vb_i.
\]
The new vectors still sum to zero.  The three vectors other than
$\vb_j$ are obtained from $\vb_i,\vb_k,\vb_l$ by a unimodular
change, so they form a basis of the same lattice.  The change in $E$ is
\[
       2\|\vb_i\|^2+2\vb_i\cdot(\vb_k+\vb_l)
                         =-2\vb_i\cdot\vb_j<0,
\]
a contradiction.  This is Selling's reduction step; see also
\cite[Appendix~A, Lemma~A.1]{Kurlin2022}.

Fix this superbase and put $p_{ij}=-\vb_i\cdot\vb_j\ge0$ for
$i\ne j$, with $p_{ij}=p_{ji}$.  If
$\vx=\sum_{i=1}^3m_i\vb_i$ and $m_0=0$, expansion gives Selling's
identity
\begin{equation}\label{eq:selling-quadratic-identity}
             \|\vx\|^2=
             \sum_{0\le i<j\le3}p_{ij}(m_i-m_j)^2.
\end{equation}
For $\sigma\in S_3$, denote the vertices of
$\Delta_\sigma(B)$, where $B=(\vb_1,\vb_2,\vb_3)$, by
\[
 \mathbf s_0=\vnull,\qquad
 \mathbf s_k=\sum_{j=1}^k\vb_{\sigma(j)}\quad(1\le k\le3).
\]
Their circumcenter $\vc_\sigma$ is the unique solution of
\begin{equation}\label{eq:kuhn-circumcenter-system}
              2\vc_\sigma\cdot\mathbf s_k=\|\mathbf s_k\|^2,
              \qquad 1\le k\le3.
\end{equation}
To verify that its circumsphere has no lattice point in its interior,
order the indices by
$\sigma(1)\prec\sigma(2)\prec\sigma(3)\prec0$ and set
\[
                 L_\sigma(\vx)=
                    \sum_{i\prec j}p_{ij}(m_i-m_j).
\]
At each $\mathbf s_k$, all the differences $m_i-m_j$ in this sum are
$0$ or $1$.  Hence \eqref{eq:selling-quadratic-identity} gives
$L_\sigma(\mathbf s_k)=\|\mathbf s_k\|^2$.
Since $\mathbf s_1,\mathbf s_2,\mathbf s_3$ are independent,
\eqref{eq:kuhn-circumcenter-system} implies
$L_\sigma(\vx)=2\vc_\sigma\cdot\vx$ for every $\vx\in\Rr^3$.
For $\vx\in\Gamma$, the numbers $m_i$ are integers, so
\begin{equation}\label{eq:kuhn-empty-sphere}
 \|\vx-\vc_\sigma\|^2-\|\vc_\sigma\|^2
   =\sum_{i\prec j}p_{ij}(m_i-m_j)(m_i-m_j-1)\ge0.
\end{equation}
Thus all four $\mathbf s_k$ are nearest lattice points to
$\vc_\sigma$.  The argument allows zero values of $p_{ij}$ and
additional lattice points on the boundary of a circumsphere.

Put $V=V(\Gamma)$.  For each $k$
and every $\boldsymbol\gamma\in\Gamma$, the nearest-point property
gives
\[
 \|\vc_\sigma-\mathbf s_k-\boldsymbol\gamma\|
       \ge\|\vc_\sigma-\mathbf s_k\|.
\]
Therefore $\vc_\sigma-\mathbf s_k\in V$.
Since $V=-V$, convexity yields
\[
                       \Delta_\sigma(B)-\vc_\sigma\subset V
                       \qquad(\sigma\in S_3).
\]
Put $H=\conv\bigcup_{\sigma\in S_3}
(\Delta_\sigma(B)-\vc_\sigma)$.
Theorem~\ref{thm:kuhn_cover} gives $H+\Gamma=\Rr^3$.
Integrating its covering multiplicity over a fundamental domain of
$\Gamma$ gives $\vol(H)\ge\det\Gamma$.
On the other hand, $H\subset V$ and $\vol(V)=\det\Gamma$.
Thus $\vol(H)=\vol(V)$, and the inclusion of full-dimensional compact
convex sets implies $H=V$.

Finally, a three-dimensional parallelohedron is affinely equivalent
to a Voronoi cell of a lattice; see
\cite{Delone1929} and \cite[Section~5]{GarberGavrilyukMagazinov2015}.
If $\vc_P$ is the center of $P$, choose an invertible linear map $A$
such that $A(P-\vc_P)=V$ for the Voronoi cell of a lattice $\Gamma$.
The image of the face-to-face tiling by $P$ is a face-to-face tiling
by $V$.  Uniqueness of that tiling, with the cell at the origin fixed,
gives $A\Lambda(P)=\Gamma$.
Pulling back the representation of $V$ proves
\eqref{eq:fedorov-kuhn-hull}, with basis and translations
\[
 (A^{-1}\vb_1,A^{-1}\vb_2,A^{-1}\vb_3),\qquad
 \mathbf t_\sigma=\vc_P-A^{-1}\vc_\sigma.
\]
\end{proof}

For an explicit instance of Proposition~\ref{prop:fedorov-kuhn},
consider the regular truncated octahedron
\[
 T=\conv\{(0,\pm1,\pm2),(0,\pm2,\pm1),
          (\pm1,0,\pm2),(\pm1,\pm2,0),
          (\pm2,0,\pm1),(\pm2,\pm1,0)\}.
\]
Set
\[
 \vb_1=(2,2,-2),\qquad \vb_2=(2,-2,2),\qquad \vb_3=(-2,2,2),
 \qquad \Lambda_T=\Z \vb_1+\Z \vb_2+\Z \vb_3.
\]
For the common basis $B=(\vb_1,\vb_2,\vb_3)$, define
$Q_j=\mathbf t_j+\Delta_{\sigma_j}(B)$ by the following table.
Here $ijk$ denotes the permutation with values $(i,j,k)$.
\[
\begin{array}{c|c|c}
 j&\sigma_j&\mathbf{t}_j\\ \hline
 1&132&(-1,-2,0)\\
 2&231&(-1,0,-2)\\
 3&312&(0,-2,-1)\\
 4&213&(-2,0,-1)\\
 5&123&(-2,-1,0)\\
 6&321&(0,-1,-2)
\end{array}
\]
For example,
\[
 Q_1=\conv\{(-1,-2,0),(1,0,-2),(-1,2,0),(1,0,2)\}.
\]
Computing the four vertices in each row gives precisely the displayed
24 vertices of $T$, each once. Hence $T=\conv\bigcup_{j=1}^6 Q_j$.

\section{Infinite Arithmetic Contact Types of Three-Dimensional Noncentrally Symmetric Minimal Covering Bodies}\label{sec:3-dim asy MBs}

We first construct a twelve-vertex noncentrally symmetric minimal covering body
that contains no complete family of translated Kuhn simplices.
We then prove Theorem~\ref{thm:infinite_asym} by an infinite family
of pairwise distinct arithmetic contact types.

\subsection{A Twelve-Vertex Minimal Covering Body}

Let $\tilde T=\conv\{\vw_1,\ldots,\vw_{12}\}$, where
\begin{align*}
 \vw_1&=(0,1,8),&\vw_2&=(1,0,8),&\vw_3&=(1,1,8),\\
 \vw_4&=(0,1,1),&\vw_5&=(1,0,1),&\vw_6&=(1,1,0),\\
 \vw_7&=(1,8,0),&\vw_8&=(1,8,1),&\vw_9&=(0,8,1),\\
 \vw_{10}&=(8,1,1),&\vw_{11}&=(8,1,0),&\vw_{12}&=(8,0,1).
\end{align*}
We consider this polytope with respect to the lattice
\[
 L_{84}=\{(k_1,k_2,k_3)\in\mathbb Z^3:
                 2k_1+9k_2+35k_3\equiv0\pmod {84}\}.
\]
The vectors $(-2,2,2),(3,-3,3),(4,3,-1)$ form a basis of $L_{84}$,
and $\det L_{84}=84$.

For the standard simplex
\[
 D=\{(x,y,z)\in\mathbb R^3:x,y,z\ge0,\ x+y+z\le1\},
\]
Dougherty and Faber proved $10D+L_{84}=\mathbb R^3$; see
\cite[arXiv version, Proposition~17]{Dougherty}.  Its density is
$\vol(10D)/\det L_{84}=125/63$.
The associated nonconvex lattice tile of volume $84$ is known as the
84-shape.  Forcade and Lamoreaux studied this configuration in
connection with local minima of lattice-simplex covering
density~\cite{Forcade2000}.
Here we verify that the displayed polytope $\tilde T$ is minimal under
inclusion for the fixed lattice $L_{84}$ and excludes a complete Kuhn
family.  Its minimality will follow from four inverted tetrahedra
identified by Dougherty and Faber.

\begin{figure}[htbp]
 \centering
 \begin{tikzpicture}[x={(0.56cm,-0.21cm)},y={(-0.56cm,-0.28cm)},
                     z={(0cm,0.66cm)},line join=round,font=\small]
  \coordinate (w1) at (0,1,8);
  \coordinate (w2) at (1,0,8);
  \coordinate (w3) at (1,1,8);
  \coordinate (w4) at (0,1,1);
  \coordinate (w5) at (1,0,1);
  \coordinate (w6) at (1,1,0);
  \coordinate (w7) at (1,8,0);
  \coordinate (w8) at (1,8,1);
  \coordinate (w9) at (0,8,1);
  \coordinate (w10) at (8,1,1);
  \coordinate (w11) at (8,1,0);
  \coordinate (w12) at (8,0,1);
  \fill[black!3] (w3)--(w8)--(w10)--cycle;
  \draw[gray,densely dashed,line width=.45pt] (w1)--(w4);
  \draw[gray,densely dashed,line width=.45pt] (w2)--(w5);
  \draw[gray,densely dashed,line width=.45pt] (w4)--(w5);
  \draw[gray,densely dashed,line width=.45pt] (w4)--(w6);
  \draw[gray,densely dashed,line width=.45pt] (w4)--(w9);
  \draw[gray,densely dashed,line width=.45pt] (w5)--(w6);
  \draw[gray,densely dashed,line width=.45pt] (w5)--(w12);
  \draw[gray,densely dashed,line width=.45pt] (w6)--(w7);
  \draw[gray,densely dashed,line width=.45pt] (w6)--(w11);
  \draw[line width=.65pt] (w1)--(w2);
  \draw[line width=.65pt] (w1)--(w3);
  \draw[line width=.65pt] (w1)--(w9);
  \draw[line width=.65pt] (w2)--(w3);
  \draw[line width=.65pt] (w2)--(w12);
  \draw[line width=.65pt] (w3)--(w8);
  \draw[line width=.65pt] (w3)--(w10);
  \draw[line width=.65pt] (w7)--(w8);
  \draw[line width=.65pt] (w7)--(w9);
  \draw[line width=.65pt] (w7)--(w11);
  \draw[line width=.65pt] (w8)--(w9);
  \draw[line width=.65pt] (w8)--(w10);
  \draw[line width=.65pt] (w10)--(w11);
  \draw[line width=.65pt] (w10)--(w12);
  \draw[line width=.65pt] (w11)--(w12);
  \fill (w1) circle (1.15pt);
  \fill (w2) circle (1.15pt);
  \fill (w3) circle (1.15pt);
  \fill (w4) circle (1.15pt);
  \fill (w5) circle (1.15pt);
  \fill (w6) circle (1.15pt);
  \fill (w7) circle (1.15pt);
  \fill (w8) circle (1.15pt);
  \fill (w9) circle (1.15pt);
  \fill (w10) circle (1.15pt);
  \fill (w11) circle (1.15pt);
  \fill (w12) circle (1.15pt);
  \node[anchor=south east,xshift=-2pt,yshift=2pt,fill=white,inner sep=.8pt] at (w1) {$\vw_{1}$};
  \node[anchor=south west,xshift=2pt,yshift=2pt,fill=white,inner sep=.8pt] at (w2) {$\vw_{2}$};
  \node[anchor=north,yshift=-4pt,fill=white,inner sep=.8pt] at (w3) {$\vw_{3}$};
  \node[anchor=east,xshift=-4pt,fill=white,inner sep=.8pt] at (w4) {$\vw_{4}$};
  \node[anchor=west,xshift=4pt,fill=white,inner sep=.8pt] at (w5) {$\vw_{5}$};
  \node[anchor=north,yshift=-3pt,fill=white,inner sep=.8pt] at (w6) {$\vw_{6}$};
  \node[anchor=north east,xshift=-2pt,yshift=-2pt,fill=white,inner sep=.8pt] at (w7) {$\vw_{7}$};
  \node[anchor=south west,xshift=2pt,yshift=2pt,fill=white,inner sep=.8pt] at (w8) {$\vw_{8}$};
  \node[anchor=south east,xshift=-3pt,yshift=2pt,fill=white,inner sep=.8pt] at (w9) {$\vw_{9}$};
  \node[anchor=south east,xshift=-2pt,yshift=2pt,fill=white,inner sep=.8pt] at (w10) {$\vw_{10}$};
  \node[anchor=north west,xshift=2pt,yshift=-3pt,fill=white,inner sep=.8pt] at (w11) {$\vw_{11}$};
  \node[anchor=south west,xshift=3pt,yshift=2pt,fill=white,inner sep=.8pt] at (w12) {$\vw_{12}$};
 \end{tikzpicture}
 \caption{The twelve-vertex polytope $\tilde T$.  Dashed edges lie on
 the far side in this projection.}
 \label{fig:84-convex-hull}
\end{figure}

\begin{proposition}\label{prop:84_shape}
The polytope $\tilde T$ is a minimal covering body with respect to
$L_{84}$ and contains no complete family of translated Kuhn simplices
for any basis of $L_{84}$.
\end{proposition}

\begin{proof}
We first group the integer translates of a unit cube by their class
modulo $L_{84}$. Whole cubes contained in $\tilde T$ account for 76
classes; pairs of pieces handle the remaining eight. For minimality,
we show that every convex subbody which still covers must contain
four inverted unit tetrahedra whose convex hull is $\tilde T$.
Finally, a count of the vertex classes excludes a complete Kuhn family.

Cutting the cube $[0,8]^3$ by the planes below gives precisely the twelve
listed vertices.  Thus
\begin{equation}\label{eq:84-facets}
\begin{gathered}
 0\le x,y,z\le8,\qquad
 1\le x+y,x+z,y+z\le9,\\
 2\le x+y+z\le10
\end{gathered}
\end{equation}
is a halfspace description of $\tilde T$.
The facets in $x=0$ and $x=8$ are triangles of areas $49/2$ and $1/2$,
respectively, so $\tilde T$ is not centrally symmetric.

\smallskip\noindent\textbf{Covering.}
Put $C=[0,1]^3$ and define
\[
 \phi(\mathbf k)=2k_1+9k_2+35k_3\pmod {84}
 \qquad(\mathbf k=(k_1,k_2,k_3)\in\mathbb Z^3).
\]
We use $\phi$ to label the $L_{84}$-translation classes of unit cubes:
two integer translates of $C$ have the same label exactly when their
translation vectors differ by an element of $L_{84}$.
We cover each class using whole cubes in $\tilde T$ or portions cut
out by its supporting planes.
By \eqref{eq:84-facets}, a whole cube $\mathbf k+C$ lies in $\tilde T$ exactly when
\begin{equation}\label{eq:84-full-cubes}
 k_i\ge0,\qquad 2\le k_1+k_2+k_3\le7,
 \qquad k_i+k_j\ge1\quad(i<j).
\end{equation}
The $98$ cubes in \eqref{eq:84-full-cubes} represent $76$ residue
classes. The cubes in \eqref{eq:84-full-cubes} fill six consecutive layers
$k_1+k_2+k_3=s$, $2\le s\le7$.
Table~\ref{tab:84-cube-layers} lists the residue classes first encountered
in each layer.  Its entries are obtained by taking
$k_1=0,\ldots,s$, $k_2=0,\ldots,s-k_1$,
$k_3=s-k_1-k_2$, and omitting the three points on the coordinate axes.
In particular, this list includes every cube in \eqref{eq:84-full-cubes}.

\begin{table}[htbp]
\centering\small
\begin{tabular}{c|c|l}
$s$&Number of cubes&New values of $\phi(\mathbf k)$\\\hline
2&3&$11,37,44$\\
3&7&$13,20,39,46,53,72,79$\\
4&12&$4,15,22,23,29,30,41,48,55,62,74,81$\\
5&18&$6,17,24,25,31,32,38,43,50,57,58,64,65,71,76,83$\\
6&25&$1,8,9,16,19,26,27,33,34,40,45,47,52,59,60,66,67,73,78,80$\\
7&33&$3,5,10,18,21,28,35,36,42,49,51,54,56,61,68,69,75,82$
\end{tabular}
\caption{The $98$ contained unit cubes represent $76$ residue classes.
A residue already occurring in a preceding layer is not repeated.}
\label{tab:84-cube-layers}
\end{table}

The missing classes are
\[
                0,2,7,12,14,63,70,77.
\]
Each is covered by two portions of unit cubes, as shown in
Table~\ref{tab:84-exceptional-cubes}.  In that table $\vu=(x,y,z)$ is a point
of $C$, and $X=x+y$, $Y=y+z$, $Z=x+z$, $S=x+y+z$.
An entry $\mathbf k:$ condition means that $\mathbf k+\vu\in\tilde T$ whenever $\vu\in C$
satisfies the condition.  Both vectors in row $r$ have residue $r$.

\begin{table}[htbp]
\centering
\begin{tabular}{c|c|c}
$r$&First portion&Second portion\\\hline
0&$(0,0,0):\ S\ge2$&$(2,5,1):\ S\le2$\\
2&$(0,1,7):\ Y\le1$&$(1,0,0):\ Y\ge1$\\
7&$(0,0,5):\ X\ge1$&$(1,6,1):\ S\le2$\\
12&$(4,2,2):\ S\le2$&$(6,0,0):\ Y\ge1$\\
14&$(0,7,1):\ Y\le1$&$(7,0,0):\ Y\ge1$\\
63&$(0,7,0):\ Z\ge1$&$(5,2,1):\ S\le2$\\
70&$(0,0,2):\ X\ge1$&$(4,3,1):\ S\le2$\\
77&$(0,0,7):\ X\ge1$&$(3,4,1):\ S\le2$
\end{tabular}
\caption{Two portions cover the unit cube in each exceptional class.}
\label{tab:84-exceptional-cubes}
\end{table}

The containments follow by substituting $\mathbf k+\vu$ in \eqref{eq:84-facets}.
The two conditions in rows $0,2,14$ are complementary.  In rows
$7,70,77$, $X<1$ implies $S<2$, since $z\le1$; row $63$ uses
$Z<1\Rightarrow S<2$, and row $12$ uses $S>2\Rightarrow Y>1$.
Thus every $\vu\in C$ is covered in every residue class.
More explicitly, write an arbitrary point as $\mathbf n+\vu$, where
$\mathbf n\in\mathbb Z^3$ and $\vu\in[0,1)^3$.  The tables supply an integer
vector $\mathbf k$ with $\phi(\mathbf k)=\phi(\mathbf n)$ and $\mathbf k+\vu\in\tilde T$.  Since
$\mathbf n-\mathbf k\in L_{84}$, this proves $\tilde T+L_{84}=\mathbb R^3$.

\smallskip\noindent\textbf{Minimality.}
Dougherty and Faber identified four inverted unit tetrahedra whose
interiors are covered only by the untranslated copy of $10D$;
see \cite[arXiv version, proof of Proposition~19]{Dougherty}.
Put
\[
 \begin{gathered}
 \mathbf c_0=(1,1,1),\quad \mathbf c_1=(8,1,1),\quad
 \mathbf c_2=(1,8,1),\quad \mathbf c_3=(1,1,8),\\
 U_i=\mathbf c_i-D\quad(0\le i\le3).
 \end{gathered}
\]
The property we use is
\begin{equation}\label{eq:84-private-tetrahedra}
 \inte U_i\cap(10D+\boldsymbol\lambda)=\varnothing
 \qquad(0\ne\boldsymbol\lambda\in L_{84},\ 0\le i\le3).
\end{equation}
We verify \eqref{eq:84-private-tetrahedra} directly from the defining
congruence of $L_{84}$. Suppose that $\mathbf p\in\inte U_i$ and
$\mathbf p-\boldsymbol\lambda\in10D$ for some
$\boldsymbol\lambda\in L_{84}$.
Write $\mathbf p=\mathbf c_i-\vu$, where $u_j>0$ and
$u_1+u_2+u_3<1$, and set
\[
              \mathbf q=\mathbf c_i-(1,1,1)-\boldsymbol\lambda.
\]
The coordinate inequalities for $\mathbf p-\boldsymbol\lambda\in10D$
give $\lambda_j\le(\mathbf c_i)_j-1$.  Its sum inequality, together
with integrality, gives
$\lambda_1+\lambda_2+\lambda_3\ge
 (\mathbf c_i)_1+(\mathbf c_i)_2+(\mathbf c_i)_3-10$.
Thus
\[
 \mathbf q\in\mathbb Z_{\ge0}^3,\qquad q_1+q_2+q_3\le7,
 \qquad \phi(\mathbf q)=\phi(\mathbf c_i-(1,1,1)).
\]
The four residues on the right, for $i=0,1,2,3$, are respectively
$0,14,63,77$.

Let $\mathbf q=(x,y,z)$ satisfy these conditions, with residue $r$.
Reduction modulo $7$ gives $2(x+y)\equiv0\pmod7$, so $x+y=0$ or $7$.
If $x+y=7$, then $z=0$ and
\[
                2x+9y+35z=63-7x\in[14,63].
\]
Among the four possible residues, this gives only $r=14$ with
$\mathbf q=(7,0,0)$ or $r=63$ with $\mathbf q=(0,7,0)$.
If $x+y=0$, then $x=y=0$ and
$5z\equiv r/7\pmod{12}$.  Multiplying by $5$ gives, respectively,
$z\equiv0,10,9,7\pmod{12}$; since $0\le z\le7$, only
$\mathbf q=(0,0,0)$ for $r=0$ and $\mathbf q=(0,0,7)$ for $r=77$
remain.  Therefore $\mathbf q=\mathbf c_i-(1,1,1)$ in every case,
which forces $\boldsymbol\lambda=\vnull$ and proves
\eqref{eq:84-private-tetrahedra}.

All sixteen defining vertices of the four $U_i$ satisfy
\eqref{eq:84-facets}, and among them are all twelve vertices of
$\tilde T$.  Consequently,
\[
                  \tilde T=\conv\bigcup_{i=0}^3 U_i\subset10D.
\]
Now let $K\subseteq\tilde T$ be a convex body with
$K+L_{84}=\mathbb R^3$.  By \eqref{eq:84-private-tetrahedra},
no point of $\inte U_i$ belongs to $K+\boldsymbol\lambda$ for a
nonzero $\boldsymbol\lambda\in L_{84}$.  Hence $\inte U_i\subset K$.
Closedness gives $U_i\subset K$ for every $i$, and convexity gives
$K=\tilde T$.

\smallskip\noindent\textbf{Exclusion of a complete Kuhn family.}
Suppose that six translated Kuhn simplices for a basis of $L_{84}$ lie
in $\tilde T$, and let $Q$ be their convex hull.
Theorem~\ref{thm:kuhn_cover} gives $Q+L_{84}=\mathbb R^3$; minimality
then gives $Q=\tilde T$.  The four vertices of each translated simplex
belong to one affine coset of $L_{84}$, even when its translation vector
is not integral.  Every vertex of $Q$ is a vertex of one of the six
simplices.  Thus its vertices meet at most six lattice cosets.
On the other hand, in the order displayed above,
\[
 (\phi(\vw_1),\ldots,\phi(\vw_{12}))
 =(37,30,39,44,37,11,74,25,23,60,25,51).
\]
There are ten distinct residue classes in this list, a contradiction.
\end{proof}

\subsection{Proof of Theorem \ref{thm:infinite_asym}: Existence of Infinitely Many Types}

For a positive integer $m$, let
\[
 T_m=\conv\{(0,0,0),(1,0,0),(0,1,0),(1,1,m)\}
\]
be the Reeve tetrahedron, and put
\[
 \tilde T_m=\conv\bigl(T_m\cup(T_m+(0,0,1))\bigr)
           =T_m+[\vnull,\ve_3],\qquad \ve_3=(0,0,1).
\]
The tetrahedron $T_m$ has volume $m/6$ and no integer points other than
its vertices: its projection is $[0,1]^2$, and the fiber over each integer
point of that square is one of the four vertices.  The eight vertices of
$\tilde T_m$ are
\[
 \begin{gathered}
 \vv_1=(0,0,0),\quad\vv_2=(1,0,0),\quad
 \vv_3=(0,1,0),\quad\vv_4=(1,1,m),\\
 \vv_{4+i}=\vv_i+\ve_3\qquad(1\le i\le4).
 \end{gathered}
\]

This construction fits \cite[Propositions~2.8 and~2.9]{Codenotti}:
$\tilde T_m$ is the convex hull of vertical unit intervals over the
four vertices of $[0,1]^2$, and every Minkowski convex combination
of these intervals has length one. Those results therefore also imply
the covering and minimality assertions below. We give a direct proof
for this family.

\begin{proposition}\label{prop:reeve_mb}
For every positive integer $m$, $\tilde T_m$ is a non-centrally-symmetric
minimal covering body with respect to $\mathbb Z^3$.
\end{proposition}

\begin{proof}
\noindent\textbf{Covering.}
The vertical fiber of $\tilde T_m$ over $(x,y)\in[0,1]^2$ is
$[a(x,y),b(x,y)+1]$, where
\begin{equation}\label{eq:reeve-fibers}
 a(x,y)=m\max\{0,x+y-1\},\qquad b(x,y)=m\min\{x,y\}.
\end{equation}
Indeed, a convex combination of the four vertices of $T_m$ assigns a
weight between $\max\{0,x+y-1\}$ and $\min\{x,y\}$ to $(1,1,m)$;
adding $[\vnull,\ve_3]$ then gives the stated fiber.
Every fiber has length at least one, and the projected square $[0,1]^2$
tiles by $\mathbb Z^2$.  After an integer horizontal translation, any
point lies above this square; an integer vertical translation then
places its height in the corresponding fiber.  Thus
$\tilde T_m+\mathbb Z^3=\mathbb R^3$.

\smallskip\noindent\textbf{Minimality.}
Let $K\subseteq\tilde T_m$ be compact and convex with
$K+\mathbb Z^3=\mathbb R^3$.
If $(x,y)\in(0,1)^2$ and $b(x,y)<z<a(x,y)+1$, then $(x,y,z)$ lies in
$\tilde T_m$ and belongs to no other integer translate of $\tilde T_m$.
Indeed, the horizontal translation must be zero because $(x,y)$ lies in
the interior of the projected unit square.  A nonzero vertical
translation would place its height outside $[a(x,y),b(x,y)+1]$, since
$z-1<a(x,y)$ and $z+1>b(x,y)+1$.
Since $K\subseteq\tilde T_m$ and $K+\mathbb Z^3$ covers the space,
every such point must belong to $K$.

For a corner $\mathbf q\in\{0,1\}^2$, put
\[
 \mathbf p_\varepsilon=(1-2\varepsilon)\mathbf q
                       +(\varepsilon,\varepsilon),
 \qquad 0<\varepsilon<\frac{1}{2(m+1)}.
\]
Here $\mathbf p_\varepsilon\in(0,1)^2$.  Since
$b(\mathbf p_\varepsilon)-a(\mathbf p_\varepsilon)=m\varepsilon$
and $(m+1)\varepsilon<1$, both points
\begin{equation}\label{eq:reeve-exclusive-points}
 \bigl(\mathbf p_\varepsilon,b(\mathbf p_\varepsilon)+\varepsilon\bigr),
 \qquad
 \bigl(\mathbf p_\varepsilon,a(\mathbf p_\varepsilon)+1-\varepsilon\bigr)
\end{equation}
satisfy the preceding strict inequalities and therefore belong to $K$;
see Figure~\ref{fig:exclusive-fiber}.
As $\varepsilon\downarrow0$, they approach the two vertices of
$\tilde T_m$ over $\mathbf q$: the functions $a,b$ are continuous and
$a(\mathbf q)=b(\mathbf q)$.
Letting $\mathbf q$ range over the four corners, closedness forces all
eight vertices into $K$, and convexity gives $K=\tilde T_m$.

\begin{figure}[htbp]
 \centering
 \begin{tikzpicture}[x=3.25cm,y=1cm,font=\small,line cap=round]
  \draw[gray,line width=1pt] (-1,.85)--(.4,.85);
  \draw[gray,line width=1pt] (1,-.85)--(2.4,-.85);
  \draw[line width=1pt] (0,0)--(1.4,0);
  \draw[line width=2.5pt] (.4,0)--(1,0);
  \draw[gray,densely dashed] (.4,.83)--(.4,-.25);
  \draw[gray,densely dashed] (1,.25)--(1,-.83);
  \foreach \x/\y in {-1/.85,.4/.85,0/0,1.4/0,1/-.85,2.4/-.85}
    \draw (\x,{\y-.055})--(\x,{\y+.055});
  \node[above=4pt] at (-1,.85) {$a-1$};
  \node[above=4pt] at (.4,.85) {$b$};
  \node[above=4pt] at (-.3,.85) {$[a-1,b]$};
  \node[above=4pt] at (0,0) {$a$};
  \node[above=4pt] at (1.4,0) {$b+1$};
  \node[below=4pt] at (1,-.85) {$a+1$};
  \node[below=4pt] at (2.4,-.85) {$b+2$};
  \node[below=4pt] at (1.7,-.85) {$[a+1,b+2]$};
  \draw[fill=white,line width=.6pt] (.4,0) circle (2pt);
  \draw[fill=white,line width=.6pt] (1,0) circle (2pt);
  \fill (.55,0) circle (1.8pt);
  \fill (.85,0) circle (1.8pt);
  \node[above=5pt,fill=white,inner sep=1pt] at (.55,0) {$b+\varepsilon$};
  \node[below=5pt,fill=white,inner sep=1pt] at (.85,0) {$a+1-\varepsilon$};
  \node at (.7,-1.95) {covered only by $[a,b+1]$};
 \end{tikzpicture}
 \caption{The fiber over $\mathbf p_\varepsilon$ and its two neighboring
 integer translates, drawn schematically at different levels for clarity.
 Here $a=a(\mathbf p_\varepsilon)$ and $b=b(\mathbf p_\varepsilon)$.
 The thick open interval $(b,a+1)$ meets no nonzero integer translate
 of $[a,b+1]$.  The two points in \eqref{eq:reeve-exclusive-points}
 are marked by their heights in this interval.}
 \label{fig:exclusive-fiber}
\end{figure}

Every integer point of $\tilde T_m$ projects to a corner of $[0,1]^2$.
The corner fibers are unit intervals with integer endpoints, so
\begin{equation}\label{eq:reeve-integral-points}
 \tilde T_m\cap\mathbb Z^3
   =\operatorname{vert}(\tilde T_m)=\{\vv_1,\ldots,\vv_8\}.
\end{equation}
Finally, a center of symmetry would project to $(1/2,1/2)$ and
interchange the fibers over opposite corners.
The fibers over $(0,0)$ and $(1,1)$ would force its third coordinate
to be $(m+1)/2$, whereas those over $(1,0)$ and $(0,1)$ would force
it to be $1/2$.  This is impossible for $m\ge1$.
\end{proof}

\begin{proof}[Proof of Theorem~\ref{thm:infinite_asym}]
By Proposition~\ref{prop:reeve_mb}, $\tilde T_m$ is a minimal covering
body for $\mathbb Z^3$ and is not centrally symmetric.
It remains to distinguish the arithmetic contact types.
All vertices of $\tilde T_m$ belong to $\mathbb Z^3$.  Hence the root
contact polytope at any vertex $\vv$ is
\[
 P_{\vv}=\conv\bigl((\tilde T_m-\vv)\cap\mathbb Z^3\bigr)
     =\tilde T_m-\vv,
\]
by \eqref{eq:reeve-integral-points}.  Adding a vertical unit segment to
$T_m$ increases its volume by the area of its projection, which is one.
Therefore
\[
             \vol(P_{\vv})=\vol(\tilde T_m)=\frac m6+1.
\]
An isomorphism of arithmetic contact profiles for the fixed lattice
$\mathbb Z^3$ preserves this volume.  Since these values are pairwise
distinct, the bodies $\tilde T_m$ have pairwise distinct arithmetic
contact types.
\end{proof}

\section{Infinite Arithmetic Contact Types of Four-Dimensional Centrally Symmetric Minimal Covering Bodies}\label{sec:four-dimensional CSM}

We now construct centrally symmetric minimal covering bodies in
$\mathbb R^4$ by symmetrizing the noncentrally symmetric bodies $\tilde T_m$.
As in the three-dimensional construction, their root contact polytopes
have increasing volume.  Define
\[
 \tilde T_m^\bullet
 =\conv\bigl((\tilde T_m\times\{0\})
          \cup(-\tilde T_m\times\{1\})\bigr).
\]
This body is centrally symmetric about $(0,0,0,1/2)$.
Translation by $-(0,0,0,1/2)$ gives an origin-symmetric body.

This is the Cayley construction with endpoint bodies $\tilde T_m$
and $-\tilde T_m$. Its covering property is equivalent to the covering
of all their Minkowski convex combinations
\cite[Lemma~2.13]{Codenotti}. Once covering is established, minimality
also follows from \cite[Proposition~2.8]{Codenotti} and
Proposition~\ref{prop:reeve_mb}. We give direct proofs of both
properties below.

\begin{proposition}\label{prop:four-dimensional_mb}
For every positive integer $m$, $\tilde T_m^\bullet$ is a centrally
symmetric minimal covering body with respect to $\mathbb Z^4$.
\end{proposition}

\begin{proof}
\noindent\textbf{Covering.}
For $0\le\lambda\le1$, identify the section at fourth coordinate
$\lambda$ with
\[
 K_\lambda
 =\{\mathbf x\in\mathbb R^3:
                  (\mathbf x,\lambda)\in\tilde T_m^\bullet\}
 =(1-\lambda)\tilde T_m+\lambda(-\tilde T_m).
\]
Its projection onto the first two coordinates is
\[
 (1-\lambda)[0,1]^2+\lambda[-1,0]^2=[-\lambda,1-\lambda]^2.
\]
Since $\tilde T_m=T_m+[\vnull,\ve_3]$, we also have
\[
 K_\lambda
 =\bigl((1-\lambda)T_m+\lambda(-T_m)\bigr)
       +[-\lambda,1-\lambda]\ve_3.
\]
The last summand is a vertical segment of length one.  The other
summand has the same projection as $K_\lambda$, so every vertical
fiber of $K_\lambda$ contains an interval of length one.
Since the projected square tiles by $\mathbb Z^2$, the fiber argument
in Proposition~\ref{prop:reeve_mb} gives
$K_\lambda+\mathbb Z^3=\mathbb R^3$ for every $\lambda\in[0,1]$.
Integer translation in the fourth coordinate now gives
$\tilde T_m^\bullet+\mathbb Z^4=\mathbb R^4$.

Let $\pi(x_1,x_2,x_3,x_4)=(x_1,x_2,x_4)$.
The section projections above give
\begin{equation}\label{eq:four-dimensional-base}
 \pi(\tilde T_m^\bullet)=D_4
       =\{(s-u,t-u,u):0\le s,t,u\le1\}.
\end{equation}
Its edge vectors $(1,0,0)$, $(0,1,0)$, and $(-1,-1,1)$ form a basis
of $\mathbb Z^3$, so $D_4$ has volume one.

\smallskip\noindent\textbf{Minimality.}
Let $K'\subseteq\tilde T_m^\bullet$ be compact and convex with
$K'+\mathbb Z^4=\mathbb R^4$, and write
\[
 K'_\lambda
 =\{\mathbf x\in\mathbb R^3:(\mathbf x,\lambda)\in K'\},
 \qquad 0\le\lambda\le1.
\]
If $0<\lambda<1$ and
$(\mathbf x,\lambda)\in K'+(\mathbf u,k)$, where
$(\mathbf u,k)\in\mathbb Z^3\times\mathbb Z$, then
$0\le\lambda-k\le1$ because $K'$ lies in the slab $0\le x_4\le1$.
Since $k$ is an integer, this forces $k=0$.  Therefore
\[
 K'_\lambda+\mathbb Z^3=\mathbb R^3,
 \qquad 0<\lambda<1.
\]

Fix $\mathbf x\in\mathbb R^3$ and choose $\lambda_j\downarrow0$ with
$0<\lambda_j<1$.  For each $j$, there is $\mathbf u_j\in\mathbb Z^3$
such that $(\mathbf x-\mathbf u_j,\lambda_j)\in K'$.
For this fixed $\mathbf x$, boundedness of $K'$ confines the vectors
$\mathbf u_j$ to a finite set.  Passing to a subsequence, we may assume
that $\mathbf u_j=\mathbf u$ is constant.  Closedness gives
$(\mathbf x-\mathbf u,0)\in K'$.
Since $\mathbf x$ was arbitrary, $K'_0+\mathbb Z^3=\mathbb R^3$.
The same argument with $\lambda_j\uparrow1$ gives
$K'_1+\mathbb Z^3=\mathbb R^3$.

Now $K'_0\subseteq\tilde T_m$ and $-K'_1\subseteq\tilde T_m$, and
both sets cover $\mathbb R^3$ by $\mathbb Z^3$ because
$-\mathbb Z^3=\mathbb Z^3$.
These compact convex covering sets have nonempty interior: otherwise
their lattice translates would lie in countably many proper affine
subspaces and could not cover $\mathbb R^3$.
Proposition~\ref{prop:reeve_mb} therefore gives
\[
 K'_0=\tilde T_m,\qquad K'_1=-\tilde T_m.
\]
Thus $K'$ contains both defining faces of $\tilde T_m^\bullet$.
By convexity,
\[
 \tilde T_m^\bullet
 =\conv\bigl((\tilde T_m\times\{0\})
             \cup(-\tilde T_m\times\{1\})\bigr)
 \subseteq K'\subseteq\tilde T_m^\bullet.
\]
Finally, every integer point of $\tilde T_m^\bullet$ has fourth
coordinate zero or one and hence lies in one of its two defining faces.
The vertices of these faces together form the vertex set of
$\tilde T_m^\bullet$.  By \eqref{eq:reeve-integral-points},
\begin{equation}\label{eq:four-dimensional-integral-points}
 \tilde T_m^\bullet\cap\mathbb Z^4
 =\operatorname{vert}(\tilde T_m^\bullet)
 =\{(\vv_i,0),(-\vv_i,1):1\le i\le8\}.
\end{equation}
\end{proof}

\begin{proof}[Proof of Theorem~\ref{thm:infinite_sym}]
By Proposition~\ref{prop:four-dimensional_mb}, $\tilde T_m^\bullet$
is a centrally symmetric minimal covering body for $\mathbb Z^4$.
It remains to distinguish the arithmetic contact types.
Using the original Reeve tetrahedron $T_m$, put
\[
 A_m=\conv\bigl((T_m\times\{0\})
             \cup((-T_m-\ve_3)\times\{1\})\bigr).
\]
Adding a unit segment in the third coordinate direction extends the
two tetrahedra to the defining layers of $\tilde T_m^\bullet$, since
$\tilde T_m=T_m+[\vnull,\ve_3]$.  Thus
\begin{equation}\label{eq:four-dimensional-unit-sum}
 \tilde T_m^\bullet=A_m+[\vnull,\mathbf E],
 \qquad \mathbf E=(0,0,1,0).
\end{equation}

After translating $A_m$ by $(0,0,1/2,-1/2)$, its eight defining points
are the four columns of the matrix below and their negatives:
\[
 \left|\det\begin{pmatrix}
 0&1&0&1\\
 0&0&1&1\\
 \tfrac12&\tfrac12&\tfrac12&m+\tfrac12\\
 -\tfrac12&-\tfrac12&-\tfrac12&-\tfrac12
 \end{pmatrix}\right|=\frac m2.
\]
These columns are linearly independent.  The translated body is
therefore a linear image of the standard four-dimensional crosspolytope.
Joining its facets to the origin partitions it into $2^4$ simplices,
each with volume $(m/2)/4!$.  Hence
\[
 \vol_4(A_m)=\frac{2^4}{4!}\cdot\frac m2=\frac m3.
\]
The projection $\pi$ in \eqref{eq:four-dimensional-base} annihilates
$\mathbf E$, so \eqref{eq:four-dimensional-unit-sum} gives
$\pi(A_m)=\pi(\tilde T_m^\bullet)=D_4$, which has volume one.
Adding $[\vnull,\mathbf E]$ lengthens every nonempty fiber by one.
Integrating the fiber lengths gives
\begin{equation}\label{eq:four-dimensional-volume}
 \vol_4(\tilde T_m^\bullet)=\vol_4(A_m)+1=1+\frac m3.
\end{equation}

All vertices of $\tilde T_m^\bullet$ are integer points.
Consequently, at any vertex $\vv$, the root contact polytope satisfies
\[
 P_{\vv}=\conv\bigl((\tilde T_m^\bullet-\vv)\cap\mathbb Z^4\bigr)
        =\tilde T_m^\bullet-\vv:
\]
it is contained in this translated body and contains all its vertices.
An isomorphism of arithmetic contact profiles for $\mathbb Z^4$ maps
these contact polytopes by a lattice automorphism, whose determinant
has absolute value one.  It therefore preserves their four-dimensional
volumes.  By \eqref{eq:four-dimensional-volume}, these volumes are
pairwise distinct as $m$ ranges over the positive integers.
Hence the bodies have pairwise distinct arithmetic contact types.
\end{proof}

Whether these bodies contain a complete family of translated
four-dimensional Kuhn simplices is not settled by this construction.
\section{A Finite Sufficient Condition for Lattice Coverings}
\label{sec:discussion}

For a prescribed lattice basis, containment of a complete translated
Kuhn family can be tested by six intersections of translates of $K$.
Theorem~\ref{thm:kuhn_cover} makes these a sufficient condition for
covering, and minimality turns containment into a representation of
$K$. Central symmetry reduces the number of tests to three.

Let $B=(\mathbf b_1,\mathbf b_2,\mathbf b_3)$ be a basis of a full-rank
lattice $\Lambda\subset\mathbb R^3$, and put
$\mathbf s=\mathbf b_1+\mathbf b_2+\mathbf b_3$. For a convex body $K$
and $\sigma\in S_3$, define
\begin{equation}\label{eq:fixed-basis-intersections}
\begin{split}
 X_\sigma(K;B)={}&K\cap(K-\mathbf b_{\sigma(1)})\\
 &{}\cap(K-\mathbf b_{\sigma(1)}-\mathbf b_{\sigma(2)})
       \cap(K-\mathbf s).
\end{split}
\end{equation}

\begin{proposition}\label{conj:intersection}
A translate of $\Delta_\sigma(B)$ is contained in $K$ if and only if
$X_\sigma(K;B)\ne\varnothing$. If these intersections are nonempty for
all $\sigma\in S_3$, then $K+\Lambda=\mathbb R^3$.
If all six intersections are nonempty and $K$ is minimal for
$\Lambda$, then for any choices
$\mathbf t_\sigma\in X_\sigma(K;B)$,
\[
 K=\conv\bigcup_{\sigma\in S_3}
       \bigl(\mathbf t_\sigma+\Delta_\sigma(B)\bigr).
\]
\end{proposition}

\begin{proof}
The condition $\mathbf t+\Delta_\sigma(B)\subset K$ is equivalent,
by convexity, to the containment of its four vertices in $K$. These
four containments say exactly that $\mathbf t\in X_\sigma(K;B)$.
Theorem~\ref{thm:kuhn_cover}, after the linear change of coordinates
sending the standard basis to $B$, shows that the convex hull of the
six chosen translates covers with $\Lambda$. This hull is contained
in $K$, and minimality forces equality when $K$ is minimal.
\end{proof}

\begin{corollary}\label{cor:symmetric-three-tests}
Let $K\subset\mathbb R^3$ be a centrally symmetric convex body. If
\[
 X_{123}(K;B)\ne\varnothing,\qquad
 X_{132}(K;B)\ne\varnothing,\qquad
 X_{213}(K;B)\ne\varnothing,
\]
then $K+\Lambda=\mathbb R^3$.
\end{corollary}

\begin{proof}
Nonemptiness of the intersections and the covering conclusion are
unchanged by translating $K$. We may therefore assume $K=-K$.
For $\sigma=(i,j,k)$, let $\sigma^{\mathrm{rev}}=(k,j,i)$.
The vertex sets give
\[
 \Delta_{\sigma^{\mathrm{rev}}}(B)
       =\mathbf s-\Delta_\sigma(B).
\]
Consequently, if $\mathbf t+\Delta_\sigma(B)\subset K=-K$, then
\[
 -\mathbf t-\mathbf s+\Delta_{\sigma^{\mathrm{rev}}}(B)\subset K.
\]
The three permutations in the statement represent the three reversal
pairs in $S_3$. Apply Proposition~\ref{conj:intersection}.
\end{proof}

For a polytope
\[
 K=\{\mathbf y\in\mathbb R^3:
       \mathbf a_r\cdot\mathbf y\le\beta_r, 1\le r\le N\},
\]
nonemptiness of $X_\sigma(K;B)$ is the feasibility of the finite system
\[
 \mathbf a_r\cdot(\mathbf t+\mathbf v)\le\beta_r
 \quad\bigl(1\le r\le N,
       \ \mathbf v\in\operatorname{Vert}\Delta_\sigma(B)\bigr).
\]
Thus each test is a linear feasibility problem with three unknowns.
For $K=-K$, the three tests are equivalent to
containment of a complete translated Kuhn family for this basis.
Problem~\ref{conj:sym_classification} asks which minimal covering
bodies admit a basis for which these tests succeed. The criterion
itself concerns a prescribed basis and supplies no bound on the
bases needed in such a search.

\appendix

\section{Proof of the Face-Pairing Lemma}
\label{app:face-pairing}

We prove Lemma~\ref{lem:periodic-face-pairing} by counting covering
tetrahedra modulo two and using continuity at degenerate parameters.

\begin{proof}[Proof of Lemma~\ref{lem:periodic-face-pairing}]
Here a labelled triangle is obtained by omitting one of the four
vertex labels; these labels are retained even when points coincide.
A tetrahedron can become degenerate only when its determinant vanishes.
Since each determinant is a nonzero polynomial in $t$, there are only
finitely many exceptional parameters at which any tetrahedron is
degenerate. All the tetrahedra for $0\le t\le1$ lie in a common bounded
set, so only finitely many of their lattice translates can meet a
fixed bounded set, uniformly in $t$.

Fix a nonexceptional parameter $t$. For a point $\mathbf y$ outside all
translated triangular faces, let $n_t(\mathbf y)$ be the number of pairs
$(j,\mathbf z)\in\{1,\ldots,N\}\times\mathbb Z^3$ for which
$\mathbf y\in\operatorname{int}(Q_j(t)+\mathbf z)$.
This number is finite by local finiteness. If paired labelled faces
satisfy $F=F'+\mathbf z_0$, then $F+\mathbf z$ is paired with
$F'+(\mathbf z+\mathbf z_0)$ for every $\mathbf z\in\mathbb Z^3$.
Thus each translated labelled face incidence belongs to exactly one
pair of coincident triangles. Along a polygonal path crossing faces
transversely and avoiding edges and intersections of noncoplanar faces,
each pair changes the count by $0$ when the two tetrahedra lie on
opposite sides, and by $\pm2$ when they lie on the same side.
If several coplanar pairs overlap at a crossing, their contributions
are added. Consequently the parity of $n_t(\mathbf y)$ is independent
of $\mathbf y$ away from the faces.

To compare different parameters, fix $t_0\in[0,1]$. Choose a point
$\mathbf y$ outside all translated triangular faces at $t_0$ and
all degenerate sets $Q_j(t_0)+\mathbf z$. Such a point exists because
these sets have dimension at most two and are locally finite.
The point may be chosen separately for each $t_0$, since at each
nonexceptional parameter the parity is independent of the spatial
point. Only finitely many lattice translates can meet a fixed
neighbourhood of $\mathbf y$, uniformly in $t$. For those finitely
many translates, membership of $\mathbf y$ is unchanged when $t$ is
sufficiently close to $t_0$: an interior point remains interior,
and a point outside a closed tetrahedron remains outside by positive
distance and continuity. In particular, degenerate tetrahedra at
$t_0$ remain disjoint from $\mathbf y$ near $t_0$. Hence the common
parity for nonexceptional parameters is locally constant and agrees
on both sides of every interior exceptional parameter.

At $t_0=0$, choose the point also from the full-measure set covered
exactly once by the nondegenerate initial tetrahedra and their lattice
translates. Its count is then 1 and remains 1 at all sufficiently close
nonexceptional parameters. Since there are only finitely many
exceptional parameters, it follows that
\begin{equation}\label{eq:kuhn-constant-parity}
                 n_t(\mathbf y)\equiv1\pmod2
\end{equation}
for every nonexceptional $t$ and every $\mathbf y$ outside the faces.

Thus every point outside the faces is covered at each nonexceptional
parameter. The union is locally finite and closed, so it covers all
of $\mathbb R^3$. At an exceptional parameter, take a sequence of
nonexceptional parameters converging to it. For a fixed point,
boundedness leaves only finitely many possible pairs $(j,\mathbf z)$;
a constant pair along a subsequence and continuity of the vertices
give coverage in the limit.
\end{proof}

\section*{Acknowledgements}

The authors thank Alexey Balitskiy for helpful discussions and for
sharing his topological argument for the higher-dimensional covering
criterion. They thank Mei Han for bringing the work of Giulia
Codenotti, Ansgar Freyer, and Katarina Krivoku\'ca to their attention,
and these three authors for constructive correspondence.

The new definitions, main results, conjectures, problems and overall proof framework were
developed by the authors. GPT was used only for language editing and
improvements to the proofs. The authors take full responsibility for
the content of the paper.

\end{document}